\documentclass[a4paper,
fontsize=11pt,%
oneside,%
numbers=enddot]{scrartcl}
\KOMAoptions{DIV=12}    
\usepackage[utf8]{inputenc}
\usepackage[T1]{fontenc}
\usepackage{textcomp}
\usepackage[fleqn]{amsmath}
\usepackage{amsthm}
\usepackage{amssymb}
\usepackage{amsfonts}
\usepackage{libertine}
\usepackage[libertine,
slantedGreek,
nosymbolsc,
nonewtxmathopt,
subscriptcorrection]{newtxmath}
\usepackage[scaled=0.95,varqu,varl]{inconsolata}
\usepackage{mathrsfs}   
\usepackage{euscript}   
\usepackage{leftindex}
\allowdisplaybreaks[1]  
\numberwithin{equation}{section}    
\usepackage[svgnames,hyperref]{xcolor}
\definecolor{dblue}{HTML}{0455BF}
\definecolor{dgreen}{HTML}{02724A}
\definecolor{dgreen2}{HTML}{025951}
\definecolor{dred}{HTML}{D90404}
\definecolor{dviolet}{HTML}{42208C}
\definecolor{labelkey}{HTML}{025951}
\definecolor{refkey}{HTML}{025951}
\addtokomafont{section}{\centering}

\usepackage{enumitem}
\setlist{itemsep=-2.0pt}

\usepackage{upref}  
\usepackage{hyperref}
\hypersetup{colorlinks=true,
linktocpage=true,
linkcolor=dblue,
citecolor=dgreen,
urlcolor=dred,
pdfencoding=auto,
hypertexnames=false}
\makeatletter
\g@addto@macro\th@plain{
\thm@headfont{\bfseries\sffamily}
\thm@notefont{\mdseries}}
\g@addto@macro\th@definition{
\thm@headfont{\bfseries\sffamily}
\thm@notefont{\mdseries}}
\g@addto@macro\th@remark{
\thm@headfont{\bfseries\sffamily}
\thm@notefont{\mdseries}}
\makeatother
\theoremstyle{plain}
\newtheorem{theorem}{Theorem}[section]
\newtheorem{proposition}[theorem]{Proposition}
\newtheorem{corollary}[theorem]{Corollary}
\newtheorem{lemma}[theorem]{Lemma}

\theoremstyle{definition}
\newtheorem{definition}[theorem]{Definition}
\newtheorem{example}[theorem]{Example}
\newtheorem{problem}[theorem]{Problem}
\newtheorem{assumption}[theorem]{Assumption}
\theoremstyle{remark}
\newtheorem{remark}[theorem]{Remark}

\usepackage[extdef=true]{delimset}	
\DeclareMathDelimiterSet{\scal}[2]{
\selectdelim[l]<{#1}
\mathpunct{}\selectdelim[p]|
{#2}\selectdelim[r]>}

\newcommand{\menge}[2]{\bigl\{{#1}\mid{#2}\bigr\}} 
\DeclareMathDelimiterSet{\Menge}[2]{\selectdelim[l]\{
{#1}\selectdelim[m]|{#2}\selectdelim[r]\}}
\newcommand*\Cdot{{\mkern 1.6mu\cdot\mkern 1.6mu}}
\makeatletter
\def\upintkern@{\mkern-7mu\mathchoice{\mkern-3.5mu}{}{}{}}
\def\upintdots@{\mathchoice{\mkern-4mu\@cdots\mkern-4mu}%
{{\cdotp}\mkern1.5mu{\cdotp}\mkern1.5mu{\cdotp}}%
{{\cdotp}\mkern1mu{\cdotp}\mkern1mu{\cdotp}}%
{{\cdotp}\mkern1mu{\cdotp}\mkern1mu{\cdotp}}}
\makeatother
\DeclareFontFamily{OMX}{mdbch}{}
\DeclareFontShape{OMX}{mdbch}{m}{n}{ <->s * [0.8]  mdbchr7v }{}
\DeclareFontShape{OMX}{mdbch}{b}{n}{ <->s * [0.8]  mdbchb7v }{}
\DeclareFontShape{OMX}{mdbch}{bx}{n}{<->ssub * mdbch/b/n}{}
\DeclareSymbolFont{uplargesymbols}{OMX}{mdbch}{m}{n}
\SetSymbolFont{uplargesymbols}{bold}{OMX}{mdbch}{b}{n}
\DeclareMathSymbol{\upintop}{\mathop}{uplargesymbols}{82}
\DeclareMathSymbol{\upointop}{\mathop}{uplargesymbols}{"48}
\makeatletter
\renewcommand{\int}{\DOTSI\upintop\ilimits@}
\renewcommand{\oint}{\DOTSI\upointop\ilimits@}
\makeatother
\DeclareFontFamily{U}{BOONDOX-calo}{\skewchar\font=45 }
\DeclareFontShape{U}{BOONDOX-calo}{m}{n}{
  <-> s*[1.05] BOONDOX-r-calo}{}
\DeclareFontShape{U}{BOONDOX-calo}{b}{n}{
  <-> s*[1.05] BOONDOX-b-calo}{}
\DeclareMathAlphabet{\mathcalboondox}{U}{BOONDOX-calo}{m}{n}
\SetMathAlphabet{\mathcalboondox}{bold}{U}{BOONDOX-calo}{b}{n}
\DeclareMathAlphabet{\mathbcalboondox}{U}{BOONDOX-calo}{b}{n}
\newcommand{\MM}{\mathcalboondox{m}}
\newcommand{\LL}{\mathcalboondox{l}}
\newcommand{\CC}{\mathcalboondox{c}}
\newcommand{\RR}{\mathbb{R}}
\newcommand{\NN}{\mathbb{N}}

\newcommand{\HH}{\mathcal{H}}
\newcommand{\GG}{\mathfrak{H}}
\newcommand{\AW}{\mathsf{A}_{\omega}}
\newcommand{\BW}{\mathsf{B}_{\omega}}
\newcommand{\HW}{\mathsf{H}_{\omega}}
\newcommand{\TW}{\mathsf{T}_{\omega}}
\newcommand{\LW}{\mathsf{L}_{\omega}}
\newcommand{\BE}{\EuScript{B}}
\newcommand{\FF}{\EuScript{F}}
\newcommand{\kut}{\boldsymbol{\EuScript{K}}}
\newcommand{\sad}{\boldsymbol{\EuScript{S}}}
\newcommand{\pinf}{{+}\infty}
\newcommand{\minf}{{-}\infty}
\newcommand{\zeroun}{\intv[o]{0}{1}}

\newcommand{\RXX}{\intv{\minf}{\pinf}}
\newcommand{\RX}{\intv[l]0{\minf}{\pinf}}
\newcommand{\RP}{\intv[r]0{0}{\pinf}}
\newcommand{\RPP}{\intv[o]0{0}{\pinf}}

\newcommand{\emp}{\varnothing}
\newcommand{\forallmu}{\forall^{\mu}}

\newcommand{\Int}{\displaystyle\int}
\newcommand{\minimize}[2]{\underset{\substack{{#1}}}
{\operatorname{minimize}}\;\;#2}

\newcommand{\pushfwd}%
{\ensuremath{\mbox{\Large$\,\triangleright\,$}}}

\DeclareMathOperator{\Argmin}{Argmin}

\newcommand{\Id}{\mathrm{Id}}
\newcommand{\moyo}[2]{\leftindex[I]^{#2}{#1}}

\DeclareMathOperator{\cran}{\overline{ran}}
\DeclareMathOperator{\ran}{ran}
\DeclareMathOperator{\cdom}{\overline{dom}}
\DeclareMathOperator{\dom}{dom}

\DeclareMathOperator{\gra}{gra}
\DeclareMathOperator{\zer}{zer}

\DeclareMathOperator{\rec}{rec}
\DeclareMathOperator{\prox}{prox}
\DeclareMathOperator{\proj}{proj}
\DeclareMathOperator{\spc}{\overline{span}}

\DeclareFontFamily{U}{mathb}{}
\DeclareFontShape{U}{mathb}{m}{n}{<-5.5> mathb5 <5.5-6.5> mathb6 
<6.5-7.5> mathb7 <7.5-8.5> mathb8 <8.5-9.5> mathb9 <9.5-11> mathb10
<11-> mathb12}{}
\DeclareSymbolFont{mathb}{U}{mathb}{m}{n}
\DeclareFontSubstitution{U}{mathb}{m}{n}
\DeclareMathSymbol{\blackdiamond}{\mathbin}{mathb}{"0C}

\renewcommand{\leq}{\leqslant}
\renewcommand{\geq}{\geqslant}
\newcommand{\exi}{\exists\,}
\newcommand{\weakly}{\rightharpoonup}
\newcommand{\mae}{\text{\normalfont$\mu$-a.e.}}

\renewenvironment{abstract}{%
\vspace*{-0.50cm}
\small
\quotation%
\noindent%
{\normalfont\bfseries\sffamily
\nobreak\abstractname\ }%
}{%
\endquotation%
\medskip
}
\renewcommand{\abstractname}{Abstract.}

\newcommand\mscsname{MSC classification.}
\newenvironment{keywords}
{\renewcommand\abstractname{\keywordsname}\begin{abstract}}
{\end{abstract}}
\newenvironment{MSC}
{\renewcommand\abstractname{\mscsname}\begin{abstract}}
{\end{abstract}}
\usepackage[auth-sc]{authblk}
\newcommand{\email}[1]{\href{mailto:#1}{\nolinkurl{#1}}}
\renewcommand*\Affilfont{\normalfont\normalsize}
\newcommand\affilcr{\protect\\ \protect\Affilfont}
\makeatletter
\renewcommand\AB@affilsepx{\protect\\[0.5em]}
\makeatother
\author[1]{Minh N. B\`ui}
\affil[1]{University of Graz
\affilcr
Department of Mathematics and Scientific Computing, NAWI Graz
\affilcr
8010 Graz, Austria
\affilcr
\email{minh.bui@uni-graz.at}
}
\author[2]{Patrick L. Combettes}
\affil[2]{North Carolina State University
\affilcr
Department of Mathematics
\affilcr
Raleigh, NC 27695, USA
\affilcr
\email{plc@math.ncsu.edu}
}

\begin{document}

\title{%
Hilbert Direct Integrals of Monotone Operators\thanks{%
Contact author: P. L. Combettes.
Email: \email{plc@math.ncsu.edu}.
Phone: +1 919 515 2671.
The work of P. L. Combettes was supported by the National
Science Foundation under grant CCF-2211123.
}}

\date{~}

\maketitle

\centerline{\textit{Dedicated to the memory of H\'edy Attouch}}

\vspace{12mm}

\begin{abstract}
Finite Cartesian products of operators play a central role in
monotone operator theory and its applications. 
Extending such products to arbitrary families of operators
acting on different Hilbert spaces is an open problem, which we
address by introducing the Hilbert direct integral of a family of
monotone operators. The
properties of this construct are studied and conditions under which
the direct integral inherits the properties of the factor operators
are provided. The question of determining whether the Hilbert direct
integral of a family of subdifferentials of convex functions is
itself a subdifferential leads us to introducing the Hilbert direct
integral of a family of functions. We establish explicit
expressions for evaluating the Legendre conjugate, subdifferential,
recession function, Moreau envelope, and proximity operator of such
integrals. Next, we propose a duality framework for monotone
inclusion problems involving integrals of linearly composed
monotone operators and show its pertinence towards the development
of numerical solution methods. Applications to inclusion and
variational problems are discussed.
\end{abstract}


\begin{MSC}
49J53, 46G10, 58E30
\end{MSC}

\newpage

\section{Introduction}
\label{sec:1}

Let $\mathsf{H}$ be a real Hilbert space with scalar product
$\scal{\Cdot}{\Cdot}_{\mathsf{H}}$ and power set $2^{\mathsf{H}}$.
An operator $\mathsf{A}\colon\mathsf{H}\to 2^{\mathsf{H}}$ is
monotone if 
\begin{equation}
\label{e:b0}
(\forall\mathsf{x}\in\mathsf{H})
(\forall\mathsf{y}\in\mathsf{H})
(\forall\mathsf{x}^*\in\mathsf{A}\mathsf{x})
(\forall\mathsf{y}^*\in\mathsf{A}\mathsf{y})
\quad
\scal{\mathsf{x}-\mathsf{y}}{
\mathsf{x}^*-\mathsf{y}^*}_{\mathsf{H}}\geq 0.
\end{equation}
Cartesian products of monotone operators are important
constructs that arise in many foundational and practical aspects
of the theory \cite{Livre1,Brez73,Brow68,Acnu24,Ghou09,Show97}.
Such products can be defined in a straightforward manner for a
finite family $(\mathsf{A}_k)_{1\leq k\leq p}$ of monotone
operators acting, respectively, on real Hilbert spaces
$(\mathsf{H}_k)_{1\leq k\leq p}$. Thus, if one denotes by 
$\HH=\mathsf{H}_1\oplus\cdots\oplus\mathsf{H}_p$ the
Hilbert direct sum of $(\mathsf{H}_k)_{1\leq k\leq p}$ and by
$x=(\mathsf{x}_1,\ldots,\mathsf{x}_p)$ a generic vector in $\HH$,
the product operator is \cite{Livre1}
\begin{equation}
\label{e:b1}
A\colon\HH\to 2^{\HH}\colon x\mapsto
\menge{x^*\in\HH}{\brk!{\forall k\in\{1,\ldots,p\}}\,\,
\mathsf{x}_k^*\in\mathsf{A}_k\mathsf{x}_k}.
\end{equation}
A fundamental instance of an infinite product arises in
\cite{Brez73} in the context of evolution
equations. There, $(\Omega,\FF,\mu)$ is a measure space,
$\mathsf{H}$ is a separable real Hilbert space,
$\mathsf{A}\colon\mathsf{H}\to 2^\mathsf{H}$ is a monotone
operator, $\HH=L^2(\Omega,\FF,\mu;\mathsf{H})$, and a product
operator is defined as
\begin{equation}
\label{e:b2}
A\colon\HH\to 2^{\HH}\colon x\mapsto
\menge{x^*\in\HH}{(\forallmu\omega\in\Omega)\,\,
x^*(\omega)\in\mathsf{A}\brk!{x(\omega)}},
\end{equation}
where, following \cite{Sch93b}, the symbol $\forallmu$ means
``for $\mu$-almost every.''
Another instance of an infinite product appears in 
\cite[Section~III.2]{Atto79} in the context of nonautonomous
evolution equations, where $\mu$ is the Lebesgue measure, 
$(\mathsf{A}_{t})_{t\in\intv{0}{T}}$ is a family of monotone
operators from $\mathsf{H}$ to $2^{\mathsf{H}}$,
$\HH=L^2(\intv{0}{T};\mathsf{H})$, and
\begin{equation}
\label{e:b3}
A\colon\HH\to 2^{\HH}\colon x\mapsto
\menge{x^*\in\HH}{\brk!{\forallmu t\in\intv{0}{T}}\,\,
x^*(t)\in\mathsf{A}_t\brk!{x(t)}}.
\end{equation}
Similar examples arise in
probability theory \cite{Bism73},
circuit theory \cite{Sepu23},
approximation theory \cite{Ibap21},
calculus of variations \cite{Ekel99},
partial differential equations \cite{Ghou09},
variational analysis \cite{Penn03},
convex analysis \cite{Rock74},
and evolution systems \cite{Show97}.
In terms of modeling, \eqref{e:b1} is limited to a finite number
of operators, \eqref{e:b2} requires that all the factor operators
be identical to $\mathsf{A}$, and \eqref{e:b3} imposes that all
the factor spaces be identical to $\mathsf{H}$ and operates with
the standard Lebesgue measure space $\intv{0}{T}$. The above
examples are not based on a common mathematical setup and the
question of defining a unifying theory for arbitrary products of
monotone operators acting on different spaces is open. This
question is not only of theoretical interest, but it is also
motivated by applications in areas such as dynamical systems,
stochastic optimization, and inverse problems. It is the objective
of the present paper to fill this gap by introducing such a
framework, studying the properties of the resulting product
operators, and exploring some of their applications. 

To support our framework, we bring into play the notion of a
direct integral of Hilbert spaces, which is an attempt to
extend Hilbert direct sums from finite families to arbitrary ones.
This construction originates in papers published around World War
II \cite{Gode49,Gode51,Kond44,vonN49}. We follow
\cite[Section~II.{\S}1]{Dixm69}.

\begin{definition}[\protect
{\cite[D\'efinition~II.{\S}1.1]{Dixm69}}]
\label{d:1}
Let $(\Omega,\FF,\mu)$ be a complete $\sigma$-finite measure space,
let $(\HW)_{\omega\in\Omega}$ be a family of real
Hilbert spaces, and let $\prod_{\omega\in\Omega}\HW$ be
the usual real vector space of mappings $x$ defined on $\Omega$
such that $(\forall\omega\in\Omega)$ $x(\omega)\in\HW$.
Suppose that $\mathfrak{G}$ is a vector subspace of
$\prod_{\omega\in\Omega}\HW$ which satisfies the following:
\begin{enumerate}[label={\normalfont[\Alph*]}]
\item
\label{d:1a}
For every $x\in\mathfrak{G}$, the function $\Omega\to\RR\colon
\omega\mapsto\norm{x(\omega)}_{\HW}$ is
$\FF$-measurable.
\item
\label{d:1b}
For every $x\in\prod_{\omega\in\Omega}\HW$,
\begin{equation}
\bigl[\;(\forall y\in\mathfrak{G})\;\;
\Omega\to\RR\colon
\omega\mapsto\scal{x(\omega)}{y(\omega)}_{\HW}
\,\,\text{is $\FF$-measurable}\;\bigr]
\quad\Rightarrow\quad x\in\mathfrak{G}.
\end{equation}
\item
\label{d:1c}
There exists a sequence $(e_n)_{n\in\NN}$ in $\mathfrak{G}$ such
that $(\forall\omega\in\Omega)$
$\spc\{e_n(\omega)\}_{n\in\NN}=\HW$.
\end{enumerate}
Then $((\HW)_{\omega\in\Omega},\mathfrak{G})$ is
an \emph{$\FF$-measurable vector field of real Hilbert spaces}.
\end{definition}

We shall operate within the framework of
\cite[Section~II.{\S}1.5]{Dixm69}, which revolves around the
following assumption.

\begin{assumption}
\label{a:1}
Let $(\Omega,\FF,\mu)$ be a complete $\sigma$-finite measure space,
let $((\HW)_{\omega\in\Omega},\mathfrak{G})$ be an
$\FF$-measurable vector field of real Hilbert spaces, and set
\begin{equation}
\label{e:4a5t}
\GG=\Menge3{x\in\mathfrak{G}}{\int_{\Omega}
\norm{x(\omega)}_{\HW}^2\mu(d\omega)<\pinf}.
\end{equation}
Let $\HH$ be the real vector space of equivalence classes of
$\mae$ equal mappings in $\GG$ equipped with the scalar product
\begin{equation}
\label{e:y9d2}
\scal{\Cdot}{\Cdot}_{\HH}\colon\HH\times\HH\to\RR\colon
(x,y)\mapsto\int_{\Omega}
\scal{x(\omega)}{y(\omega)}_{\HW}\mu(d\omega),
\end{equation}
where we adopt the common practice of designating by $x$ both an
equivalence class in $\HH$ and a representative of it in $\GG$.
Then $\HH$ is a Hilbert space
\cite[Proposition~II.{\S}1.5(i)]{Dixm69}, called
the \emph{Hilbert direct integral of
$(\HW)_{\omega\in\Omega}$ relative to
$\mathfrak{G}$}.
Following \cite[D\'efinition~II.{\S}1.3]{Dixm69}, we write
\begin{equation}
\label{e:n5rt}
\HH=\leftindex^{\mathfrak{G}}{\int}_{\Omega}^{\oplus}
\HW\mu(d\omega).
\end{equation}
\end{assumption}

We are now in a position to propose a definition for an arbitrary
product of set-valued operators acting on different Hilbert
spaces.

\begin{definition}
\label{d:6}
Suppose that Assumption~\ref{a:1} is in force and,
for every $\omega\in\Omega$, let
$\AW\colon\HW\to 2^{\HW}$. The \emph{Hilbert
direct integral of the operators
$(\mathsf{A}_\omega)_{\omega\in\Omega}$ relative to
$\mathfrak{G}$} is 
\begin{equation}
\label{e:b5}
\leftindex^{\mathfrak{G}}{\int}_{\Omega}^{\oplus}
\AW\mu(d\omega)\colon\HH\to 2^{\HH}\colon x\mapsto
\menge{x^*\in\HH}{(\forallmu\omega\in\Omega)\,\,
x^*(\omega)\in\AW\brk!{x(\omega)}}.
\end{equation}
\end{definition}

In tandem with Definition~\ref{d:6}, we introduce the following
notion of an arbitrary direct sum of functions defined
on different Hilbert spaces. In the convex case,
the subdifferential operator will serve as a bridge between
Definitions~\ref{d:6} and \ref{d:5}. Indeed, we shall establish
in Theorem~\ref{t:2} that, under suitable assumptions,
\begin{equation}
\label{e:b6}
\partial\brk3{\leftindex^{\mathfrak{G}}{\int}_{\Omega}^{\oplus}
\mathsf{f}_{\omega}\mu(d\omega)}
=\leftindex^{\mathfrak{G}}{\Int}_{\Omega}^{\oplus}
\partial\mathsf{f}_{\omega}\mu(d\omega).
\end{equation}

\begin{definition}
\label{d:5}
Suppose that Assumption~\ref{a:1} is in force and, for every
$\omega\in\Omega$, let
$\mathsf{f}_{\omega}\colon\HW\to\RXX$.
Suppose that, for every $x\in\GG$, the function
$\Omega\to\RXX\colon\omega\mapsto\mathsf{f}_{\omega}(x(\omega))$
is $\FF$-measurable. The \emph{{Hilbert direct integral of the
functions $(\mathsf{f}_{\omega})_{\omega\in\Omega}$ relative to
$\mathfrak{G}$}} is
\begin{equation}
\leftindex^{\mathfrak{G}}{\int}_{\Omega}^{\oplus}
\mathsf{f}_{\omega}\mu(d\omega)\colon
\HH\to\RXX\colon x\mapsto
\int_{\Omega}\mathsf{f}_{\omega}\brk!{x(\omega)}
\mu(d\omega),
\end{equation}
where we adopt the customary convention that
the integral $\int_\Omega\vartheta d\mu$ of an
$\FF$-measurable function $\vartheta\colon\Omega\to\RXX$ is the
usual Lebesgue integral, except when the Lebesgue integral
$\int_\Omega\max\{\vartheta,0\}d\mu$ is $\pinf$, in which case
$\int_\Omega\vartheta d\mu=\pinf$.
\end{definition}

The remainder of the paper is as follows. Section~\ref{sec:2}
presents our notation and provides preliminary results. The Hilbert
direct integral of a family of set-valued operators introduced in
Definition~\ref{d:6} is studied in Section~\ref{sec:3}. In
particular, we establish conditions under which properties such as
monotonicity, maximal monotonicity, cocoercivity, and averagedness
are transferable from the factor operators to the Hilbert direct
integral. We also establish formulas for the domain, range,
inverse, resolvent, and Yosida approximation of this integral.
Section~\ref{sec:4} focuses on the Hilbert direct integral of
functions of Definition~\ref{d:5}. We provide conditions for
evaluating the Legendre conjugate, the subdifferential, the
recession function, the Moreau envelope, and the proximity operator
of the Hilbert direct integral of a family of functions by applying
these operations to each factor and then taking the Hilbert direct
integral of the resulting family. In Section~\ref{sec:5}, the
results of Section~\ref{sec:3} are used to investigate integral
inclusion problems involving a family of linearly composed
monotone operators. In this context, we propose a duality theory
and discuss some applications.

\section{Notation and theoretical tools}
\label{sec:2}

\subsection{Notation}

We follow the notation of \cite{Livre1}, to which we refer for a
detailed account of the following notions.

Let $\HH$ be a real Hilbert space with identity operator
$\Id_{\HH}$, scalar product $\scal{\Cdot}{\Cdot}_{\HH}$, and 
associated norm $\norm{\Cdot}_{\HH}$. The weak convergence of a
sequence $(x_n)_{n\in\NN}$ to $x$ is denoted by $x_n\weakly x$,
and $x_n\to x$ denotes its strong convergence.

Let $C$ be a nonempty closed convex subset of $\HH$.
Then $\iota_C$ is the indicator function of $C$,
$d_C$ is the distance function to $C$,
$\proj_C$ is the projection operator onto $C$,
$C^\ominus$ is the polar cone of $C$,
and $N_C$ is the normal cone operator of $C$.

Let $T\colon\HH\to\HH$ and $\tau\in\RPP$. Then $T$ is nonexpansive
if it is $1$-Lipschitzian, 
$\tau$-averaged if $\tau\in\zeroun$ and
$\Id_{\HH}+\tau^{-1}(T-\Id_{\HH})$
is nonexpansive, $\tau$-cocoercive if 
\begin{equation}
\label{e:coco}
(\forall x\in\HH)(\forall y\in\HH)\quad
\scal{x-y}{Tx-Ty}_{\HH}\geq\tau\norm{Tx-Ty}_{\HH}^2,
\end{equation}
and firmly nonexpansive if it is $1$-cocoercive.

Let $A\colon\HH\to 2^{\HH}$. The domain of $A$ is
$\dom A=\menge{x\in\HH}{Ax\neq\emp}$, the range of $A$ is
$\ran A=\bigcup_{x\in\dom A}Ax$, the set of zeros of $A$ is
$\zer A=\menge{x\in\HH}{0\in Ax}$, and the graph of $A$ is 
$\gra A=\menge{(x,x^*)\in\HH\times\HH}{x^*\in Ax}$. The inverse of
$A$ is the operator $A^{-1}\colon\HH\to 2^{\HH}$ with graph
$\gra A^{-1}=\menge{(x^*,x)\in\HH\times\HH}{x^*\in Ax}$.
The resolvent of $A$ is $J_A=(\Id_{\HH}+A)^{-1}$,
and the Yosida approximation of $A$ of index $\gamma\in\RPP$ is
$\moyo{A}{\gamma}
=(\Id_{\HH}-J_{\gamma A})/\gamma
=(\gamma\Id_{\HH}+A^{-1})^{-1}$.
Suppose that $A$ is monotone (see \eqref{e:b0}).
Then $A$ is maximally monotone if there exists no
monotone operator $B\colon\HH\to 2^{\HH}$ such that 
$\gra A\subset\gra B\neq\gra A$. In this case, 
$\dom J_A=\HH$, $J_A$ is firmly nonexpansive, and
for every $x\in\dom A$, $Ax$ is nonempty, closed, and convex,
and we set $\moyo{A}{0}x=\proj_{Ax}0$.

We denote by $\Gamma_0(\HH)$ the class of functions 
$f\colon\HH\to\RX$ which are lower semicontinuous, convex, and 
such that $\dom f=\menge{x\in\HH}{f(x)<\pinf}\neq\emp$.
Let $f\in\Gamma_0(\HH)$. The conjugate of $f$ is 
$\Gamma_0(\HH)\ni f^*\colon x^*\mapsto
\sup_{x\in\HH}(\scal{x}{x^*}_{\HH}-f(x))$
and the subdifferential of $f$ is the maximally monotone operator
\begin{equation}
\label{e:subdiff}
\partial f\colon\HH\to 2^{\HH}\colon x\mapsto\menge{x^*\in\HH}
{(\forall y\in\HH)\,\,\scal{y-x}{x^*}_{\HH}+f(x)\leq f(y)}.
\end{equation}
The proximity operator $\prox_f=J_{\partial f}$ of $f$ maps every
$x\in\HH$ to the unique minimizer of the function 
$\HH\to\RX\colon y\mapsto f(y)+\norm{x-y}_{\HH}^2/2$,
the Moreau envelope of $f$ of index $\gamma\in\RPP$ is 
$\moyo{f}{\gamma}\colon\HH\to\RR\colon x\mapsto
\min_{y\in\HH}(f(y)+\norm{x-y}_{\HH}^2/(2\gamma))$,
and $\rec f$ is the recession function of $f$.

\subsection{Integrals of set-valued mappings}

Let $(\Omega,\FF,\mu)$ be a complete $\sigma$-finite measure space
and let $\mathsf{H}$ be a separable real Hilbert space.
For every $p\in\intv[r]{1}{\pinf}$, set
\begin{equation}
\mathscr{L}^p\brk!{\Omega,\FF,\mu;\mathsf{H}}
=\Menge3{x\colon\Omega\to\mathsf{H}}{x\,\,
\text{is $(\FF,\BE_{\mathsf{H}})$-measurable and}\,\,
\int_{\Omega}\norm{x(\omega)}_{\mathsf{H}}^p\,\mu(d\omega)<\pinf},
\end{equation}
where $\BE_{\mathsf{H}}$ stands for the Borel $\sigma$-algebra of
$\mathsf{H}$. The Lebesgue (also called Bochner \cite{Hyto16})
integral of a mapping
$x\in\mathscr{L}^1(\Omega,\FF,\mu;\mathsf{H})$ is denoted by
$\int_{\Omega}x(\omega)\mu(d\omega)$. We denote by
$L^p\brk{\Omega,\FF,\mu;\mathsf{H}}$ the space of equivalence
classes of $\mae$ equal mappings in
$\mathscr{L}^p\brk{\Omega,\FF,\mu;\mathsf{H}}$; see
\cite[Section~V.{\S}7]{Sch93b} for background. The Aumann integral
of a set-valued mapping $X\colon\Omega\to 2^{\mathsf{H}}$ is 
\begin{equation}
\label{e:aum3}
\int_{\Omega}X(\omega)\mu(d\omega)=
\Menge3{\int_{\Omega}x(\omega)\mu(d\omega)}{
x\in\mathscr{L}^1\brk!{\Omega,\FF,\mu;\mathsf{H}}\,\,
\text{and}\,\,(\forallmu\omega\in\Omega)\,\,
x(\omega)\in X(\omega)}.
\end{equation}

\subsection{Hilbert direct integrals of Hilbert spaces}
\label{sec:JD}

Going back to Definition~\ref{d:1} and Assumption~\ref{a:1},
the following examples of Hilbert direct integrals will be used
repeatedly.

\begin{example}
\label{ex:1}
Here are instances of measurable vector fields and Hilbert direct
integrals based on
\cite[Exemples on pp.~142--143 and 148]{Dixm69}.
\begin{enumerate}
\item
\label{ex:1i+}
Let $p\in\NN\smallsetminus\{0\}$ and let
$(\alpha_k)_{1\leq k\leq p}\in\RPP^p$. Set
\begin{equation}
\Omega=\set{1,\ldots,p},\quad
\FF=2^{\set{1,\ldots,p}},\quad\text{and}\quad
\brk!{\forall k\in\set{1,\ldots,p}}\;\;
\mu\brk!{\set{k}}=\alpha_k.
\end{equation}
Let $(\mathsf{H}_k)_{1\leq k\leq p}$ be separable real Hilbert
spaces and let
$\mathfrak{G}=\mathsf{H}_1\times\cdots\times\mathsf{H}_p$
be the usual Cartesian product vector space.
Then $((\mathsf{H}_k)_{1\leq k\leq p},\mathfrak{G})$ is an
$\FF$-measurable vector field of real Hilbert spaces
and $\leftindex^{\mathfrak{G}}{\int}_{\Omega}^{\oplus}
\HW\mu(d\omega)$ is the weighted Hilbert direct sum
of $(\mathsf{H}_k)_{1\leq k\leq p}$, that is, the Hilbert space
obtained by equipping $\mathfrak{G}$
with the scalar product
\begin{equation}
\brk!{(\mathsf{x}_k)_{1\leq k\leq p},
(\mathsf{y}_k)_{1\leq k\leq p}}\mapsto
\sum_{k=1}^p\alpha_k
\scal{\mathsf{x}_k}{\mathsf{y}_k}_{\mathsf{H}_k}.
\end{equation}
\item
\label{ex:1i}
In the setting of \ref{ex:1i+}, suppose that
$(\forall k\in\set{1,\ldots,p})$ $\alpha_k=1$. Then
\begin{equation}
\leftindex^{\mathfrak{G}}{\int}_{\Omega}^{\oplus}
\HW\mu(d\omega)=\mathsf{H}_1\oplus\cdots\oplus\mathsf{H}_p
\end{equation}
is the standard Hilbert direct sum of
$(\mathsf{H}_k)_{1\leq k\leq p}$.
\item
\label{ex:1ii}
Let $(\alpha_k)_{k\in\NN}$ be a sequence in $\RPP$ and set
\begin{equation}
\Omega=\NN,\quad\FF=2^{\NN},\quad\text{and}\quad
(\forall k\in\NN)\;\;\mu\brk!{\set{k}}=\alpha_k.
\end{equation}
Let $(\mathsf{H}_k)_{k\in\NN}$ be
separable real Hilbert spaces and set
$\mathfrak{G}=\prod_{k\in\NN}\mathsf{H}_k$.
Then $((\mathsf{H}_k)_{k\in\NN},\mathfrak{G})$ is an
$\FF$-measurable vector field of real Hilbert spaces and
$\leftindex^{\mathfrak{G}}{\int}_{\Omega}^{\oplus}
\HW\mu(d\omega)$ is the Hilbert space obtained
by equipping the vector space
\begin{equation}
\mathfrak{H}=\Menge3{(\mathsf{x}_k)_{k\in\NN}
\in\mathfrak{G}}{\sum_{k\in\NN}\alpha_k
\norm{\mathsf{x}_k}_{\mathsf{H}_k}^2<\pinf}
\end{equation}
with the scalar product
\begin{equation}
\brk!{(\mathsf{x}_k)_{k\in\NN},(\mathsf{y}_k)_{k\in\NN}}
\mapsto\sum_{k\in\NN}\alpha_k
\scal{\mathsf{x}_k}{\mathsf{y}_k}_{\mathsf{H}_k}.
\end{equation}
\item
\label{ex:1iii}
Let $(\Omega,\FF,\mu)$ be a complete $\sigma$-finite measure space,
let $\mathsf{H}$ be a separable real Hilbert space, and set
\begin{equation}
\brk[s]!{\;(\forall\omega\in\Omega)\;\;\HW=\mathsf{H}\;}
\quad\text{and}\quad
\mathfrak{G}=\menge{x\colon\Omega\to\mathsf{H}}{
x\,\,\text{is $(\FF,\BE_{\mathsf{H}})$-measurable}}.
\end{equation}
Then $((\HW)_{\omega\in\Omega},\mathfrak{G})$ is an
$\FF$-measurable vector field of real Hilbert spaces and
\begin{equation}
\leftindex^{\mathfrak{G}}{\int}_{\Omega}^{\oplus}
\HW\mu(d\omega)
=L^2\brk!{\Omega,\FF,\mu;\mathsf{H}}.
\end{equation}
\end{enumerate}
\end{example}

The following results are given as remarks in
\cite[Section~II.{\S}1.3]{Dixm69}. We provide proofs for
completeness.

\begin{lemma}
\label{l:1}
Let $(\Omega,\FF,\mu)$ be a complete $\sigma$-finite measure space
and let $((\HW)_{\omega\in\Omega},\mathfrak{G})$ be
an $\FF$-measurable vector field of Hilbert spaces. Then the
following hold:
\begin{enumerate}
\item
\label{l:1i}
Let $x$ and $y$ be in $\mathfrak{G}$. Then the function
$\Omega\to\RR\colon\omega\mapsto
\scal{x(\omega)}{y(\omega)}_{\HW}$ is $\FF$-measurable.
\item
\label{l:1ii}
Let $x\in\prod_{\omega\in\Omega}\HW$
and $y\in\mathfrak{G}$ be such that $x=y$ $\mae$ Then
$x\in\mathfrak{G}$.
\item
\label{l:1iii}
Let $\xi\colon\Omega\to\RR$ be $\FF$-measurable and
let $x\in\mathfrak{G}$. Then the mapping
$\xi x\colon\omega\mapsto\xi(\omega)x(\omega)$ lies in
$\mathfrak{G}$.
\item
\label{l:1iv}
Let $(x_n)_{n\in\NN}$ be a sequence in $\mathfrak{G}$ and
let $x\in\prod_{\omega\in\Omega}\HW$.
Suppose that $(\forallmu\omega\in\Omega)$ $x_n(\omega)\weakly
x(\omega)$. Then $x\in\mathfrak{G}$.
\item
\label{l:1v}
There exists a sequence $(u_n)_{n\in\NN}$ in $\mathfrak{G}$
such that
\begin{equation}
\begin{cases}
(\forall n\in\NN)\;\;
\Int_{\Omega}\norm{u_n(\omega)}_{\HW}^2
\mu(d\omega)<\pinf
\medskip\\
(\forall\omega\in\Omega)\;\;
\overline{\set!{u_n(\omega)}_{n\in\NN}}=\HW.
\end{cases}
\end{equation}
\end{enumerate}
\end{lemma}
\begin{proof}
\ref{l:1i}:
Since $\mathfrak{G}$ is a vector subspace of
$\prod_{\omega\in\Omega}\HW$, $x+y\in\mathfrak{G}$ and
$x-y\in\mathfrak{G}$. Hence, by property~\ref{d:1a} in
Definition~\ref{d:1}, the functions
$\Omega\to\RR\colon\omega\mapsto
\norm{x(\omega)+y(\omega)}_{\HW}$ and
$\Omega\to\RR\colon\omega\mapsto
\norm{x(\omega)-y(\omega)}_{\HW}$
are $\FF$-measurable. Therefore, the assertion follows from the
polarization identity $(\forall\omega\in\Omega)$
$4\scal{x(\omega)}{y(\omega)}_{\HW}
=\norm{x(\omega)+y(\omega)}_{\HW}^2
-\norm{x(\omega)-y(\omega)}_{\HW}^2$.

\ref{l:1ii}:
Take $z\in\mathfrak{G}$. Then $(\forallmu\omega\in\Omega)$
$\scal{x(\omega)}{z(\omega)}_{\HW}=
\scal{y(\omega)}{z(\omega)}_{\HW}$.
At the same time, since
$y$ and $z$ lie in $\mathfrak{G}$, we deduce from
\ref{l:1i} that the function $\Omega\to\RR\colon\omega\mapsto
\scal{y(\omega)}{z(\omega)}_{\HW}$ is
$\FF$-measurable. Hence, the completeness of
$(\Omega,\FF,\mu)$ implies that the function
$\Omega\to\RR\colon\omega\mapsto
\scal{x(\omega)}{z(\omega)}_{\HW}$ is also
$\FF$-measurable. Consequently, property \ref{d:1b} in
Definition~\ref{d:1} forces $x\in\mathfrak{G}$.

\ref{l:1iii}:
We have $\xi x\in\prod_{\omega\in\Omega}\HW$.
On the other hand, for every $y\in\mathfrak{G}$, it results from
\ref{l:1i} that the function $\omega\mapsto
\scal{\xi(\omega)x(\omega)}{y(\omega)}_{\HW}
=\xi(\omega)\scal{x(\omega)}{y(\omega)}_{\HW}$ is
$\FF$-measurable. Hence, we conclude via property~\ref{d:1b} in
Definition~\ref{d:1} that $\xi x\in\mathfrak{G}$.

\ref{l:1iv}:
Let $\Xi\in\FF$ be such that $\mu(\Xi)=0$ and
$(\forall\omega\in\complement\Xi)$ $x_n(\omega)\weakly x(\omega)$.
Moreover, set
\begin{equation}
\bigl[\;(\forall n\in\NN)\;\;
y_n=1_{\complement\Xi}x_n\;\bigr]\quad\text{and}\quad
y=1_{\complement\Xi}x,
\end{equation}
and let $z\in\mathfrak{G}$. For every $n\in\NN$, it results from
\ref{l:1iii} that $y_n\in\mathfrak{G}$ and, in turn, from
\ref{l:1i} that the function $\Omega\to\RR\colon\omega\mapsto
\scal{y_n(\omega)}{z(\omega)}_{\HW}$
is $\FF$-measurable. Additionally,
\begin{equation}
(\forall\omega\in\Xi)\quad
\lim\scal{y_n(\omega)}{z(\omega)}_{\HW}
=0
=\scal{y(\omega)}{z(\omega)}_{\HW}
\end{equation}
and
\begin{equation}
\bigl(\forall\omega\in\complement\Xi\bigr)\quad
\lim\scal{y_n(\omega)}{z(\omega)}_{\HW}
=\lim\scal{x_n(\omega)}{z(\omega)}_{\HW}
=\scal{x(\omega)}{z(\omega)}_{\HW}
=\scal{y(\omega)}{z(\omega)}_{\HW}.
\end{equation}
Hence, the function
$\Omega\to\RR\colon\omega\mapsto
\scal{y(\omega)}{z(\omega)}_{\HW}$
is $\FF$-measurable as the pointwise limit of a sequence of
$\FF$-measurable functions. Therefore, appealing to
property~\ref{d:1b} in Definition~\ref{d:1}, we deduce that
$y\in\mathfrak{G}$. Consequently, since $x=y$ $\mae$, \ref{l:1ii}
yields $x\in\mathfrak{G}$.

\ref{l:1v}:
Property~\ref{d:1c} in Definition~\ref{d:1} guarantees the
existence of a sequence $(e_n)_{n\in\NN}$ in $\mathfrak{G}$ such
that $(\forall\omega\in\Omega)$
$\spc\set{e_n(\omega)}_{n\in\NN}=\HW$. Now let $(r_n)_{n\in\NN}$
be an enumeration of the set
\begin{equation}
\Menge4{\sum_{k=0}^n\alpha_ke_k}{n\in\NN\,\,\text{and}\,\,
(\alpha_k)_{0\leq k\leq n}\in\mathbb{Q}^{n+1}}.
\end{equation}
Then
\begin{equation}
\label{e:bl4x}
(\forall n\in\NN)\quad r_n\in\mathfrak{G}
\end{equation}
and
\begin{equation}
\label{e:pqht}
(\forall\omega\in\Omega)\quad
\overline{\bigl\{r_n(\omega)\bigr\}_{n\in\NN}}=\HW.
\end{equation}
Since $(\Omega,\FF,\mu)$ is $\sigma$-finite, we obtain an
increasing sequence $(\Omega_k)_{k\in\NN}$ in $\FF$
of finite $\mu$-measure such that
$\bigcup_{k\in\NN}\Omega_k=\Omega$. Set
\begin{equation}
\label{e:w9hr}
(\forall n\in\NN)(\forall m\in\NN)(\forall k\in\NN)\quad
\Xi_{n,m,k}=\menge{\omega\in\Omega_k}{
\norm{r_n(\omega)}_{\HW}\leq m}
\quad\text{and}\quad
s_{n,m,k}=1_{\Xi_{n,m,k}}r_n.
\end{equation}
For every $n\in\NN$, it results from \eqref{e:bl4x}
and property~\ref{d:1a} in Definition~\ref{d:1} that the
function $\Omega\to\RR\colon\omega\mapsto
\norm{r_n(\omega)}_{\HW}$ is
$\FF$-measurable. Therefore,
for every $n\in\NN$, every $m\in\NN$, and every $k\in\NN$,
$\Xi_{n,m,k}\in\FF$ and we thus infer from \ref{l:1iii}
and \eqref{e:bl4x} that $s_{n,m,k}\in\mathfrak{G}$ whereas,
by \eqref{e:w9hr},
\begin{equation}
\int_{\Omega}\norm{s_{n,m,k}(\omega)}_{\HW}^2
\mu(d\omega)\leq\mu(\Xi_{n,m,k})m
\leq\mu(\Omega_k)m
<\pinf.
\end{equation}
Next, take $\omega\in\Omega$, $\mathsf{x}\in\HW$, and
$\varepsilon\in\zeroun$. By \eqref{e:pqht}, there exists
$n\in\NN$ such that $\norm{r_n(\omega)-
\mathsf{x}}_{\HW}\leq\varepsilon$.
In turn, the triangle inequality gives
$\norm{r_n(\omega)}_{\HW}
\leq\varepsilon+\norm{\mathsf{x}}_{\HW}$.
However, since $\bigcup_{k\in\NN}\Omega_k=\Omega$,
there exists $k\in\NN$ such that $\omega\in\Omega_k$.
Therefore, upon choosing $m\in\NN$ such that
$m\geq\varepsilon+\norm{\mathsf{x}}_{\HW}$,
we deduce that $\omega\in\Xi_{n,m,k}$. Thus, combining with
\eqref{e:w9hr} yields
$\norm{s_{n,m,k}(\omega)-\mathsf{x}}_{\HW}
=\norm{r_n(\omega)-\mathsf{x}}_{\HW}
\leq\varepsilon$. 
\end{proof}

\begin{lemma}[\protect{\cite[Proposition~II.{\S}1.5(ii)]{Dixm69}}]
\label{l:2}
Suppose that Assumption~\ref{a:1} is in force
and let $(x_n)_{n\in\NN}$ be a sequence in $\HH$ which converges
strongly to a point $x\in\HH$. Then there exists a strictly
increasing sequence $(k_n)_{n\in\NN}$ in $\NN$ such that
$(\forallmu\omega\in\Omega)$ $x_{k_n}(\omega)\to x(\omega)$.
\end{lemma}

\section{Hilbert direct integrals of set-valued operators}
\label{sec:3}

We study the properties of the Hilbert direct integrals
of set-valued operators introduced in Definition~\ref{d:6}.
Let us first point out an important special case of
Definition~\ref{d:6}.

\begin{definition}
\label{d:3}
Suppose that Assumption~\ref{a:1} is in force and,
for every $\omega\in\Omega$, let $\mathsf{C}_{\omega}$ be a subset
of $\HW$. The \emph{Hilbert direct integral of the
sets $(\mathsf{C}_{\omega})_{\omega\in\Omega}$ relative to
$\mathfrak{G}$} is
\begin{equation}
\leftindex^{\mathfrak{G}}{\int}_{\Omega}^{\oplus}
\mathsf{C}_{\omega}\mu(d\omega)=
\menge{x\in\HH}{(\forallmu\omega\in\Omega)\,\,
x(\omega)\in\mathsf{C}_{\omega}}.
\end{equation}
\end{definition}

We first record the following facts, which are direct consequences
of Definitions~\ref{d:6} and \ref{d:3}.

\begin{proposition}
\label{p:1}
Suppose that Assumption~\ref{a:1} is in force and, for every
$\omega\in\Omega$, let $\AW\colon\HW\to 2^{\HW}$
be a set-valued operator. Set
\begin{equation}
A=\leftindex^{\mathfrak{G}}{\int}_{\Omega}^{\oplus}
\AW\mu(d\omega).
\end{equation}
Then the following hold:
\begin{enumerate}
\item
\label{p:1i}
$\dom A=\menge{x\in\HH}{(\exi x^*\in\GG)
(\forallmu\omega\in\Omega)\,\,x^*(\omega)\in\AW(x(\omega))}$.
\item
\label{p:1ii}
$\ran A=\menge{x^*\in\HH}{(\exi x\in\GG)
(\forallmu\omega\in\Omega)\,\,x^*(\omega)\in\AW(x(\omega))}$.
\item
\label{p:1iii}
$\zer A=\leftindex^{\mathfrak{G}}{\Int}_{\Omega}^{\oplus}
\zer\AW\,\mu(d\omega)$.
\item
\label{p:1iv}
$A^{-1}=\leftindex^{\mathfrak{G}}{\Int}_{\Omega}^{\oplus}
\AW^{-1}\mu(d\omega)$.
\item
\label{p:1v}
Suppose that, for every $\omega\in\Omega$, $\AW$ is
monotone. Then $A$ is monotone.
\end{enumerate}
\end{proposition}

\begin{remark}
\label{r:1}
Regarding Proposition~\ref{p:1}\ref{p:1i}, consider the setting of
Example~\ref{ex:1}\ref{ex:1ii} and suppose that, in addition, 
$(\forall k\in\NN)$ $\mathsf{H}_k=\RR$. For every $k\in\NN$, set
$\mathsf{A}_k\colon\mathsf{H}_k\to\mathsf{H}_k
\colon\mathsf{x}\mapsto k/\sqrt{\alpha_k}$. Then
\begin{equation}
\dom\brk3{\leftindex^{\mathfrak{G}}{\int}_{\Omega}^{\oplus}
\AW\mu(d\omega)}=\emp.
\end{equation}
\end{remark}

The following result examines the interplay between the properties
of the direct integral and those of its factor operators.

\begin{proposition}
\label{p:2}
Suppose that Assumption~\ref{a:1} is in force and, for every
$\omega\in\Omega$, let $\TW\colon\HW\to\HW$ be
sequentially strong-to-weak continuous. Set
\begin{equation}
T=\leftindex^{\mathfrak{G}}{\int}_{\Omega}^{\oplus}
\TW\mu(d\omega)
\end{equation}
and suppose that the following are satisfied:
\begin{enumerate}[label={\normalfont[\Alph*]}]
\item
\label{p:2A}
For every $x\in\GG$, the mapping $\omega\mapsto\TW(x(\omega))$
lies in $\mathfrak{G}$.
\item
\label{p:2B}
There exists $z\in\GG$ such that the mapping
$\omega\mapsto\TW(z(\omega))$ lies in $\GG$.
\end{enumerate}
Then the following hold:
\begin{enumerate}
\item
\label{p:2i}
Let $\beta\in\RP$. Then the following are equivalent:
\begin{enumerate}
\item
\label{p:2ia}
For $\mu$-almost every $\omega\in\Omega$, $\TW$ is
$\beta$-Lipschitzian.
\item
\label{p:2ib}
$\dom T=\HH$ and $T$ is $\beta$-Lipschitzian.
\end{enumerate}
\item
\label{p:2ii}
Let $\tau\in\RPP$. Then the following are equivalent:
\begin{enumerate}
\item
\label{p:2iia}
For $\mu$-almost every $\omega\in\Omega$, $\TW$ is
$\tau$-cocoercive.
\item
\label{p:2iib}
$\dom T=\HH$ and $T$ is $\tau$-cocoercive.
\end{enumerate}
\item
\label{p:2iii}
Let $\alpha\in\zeroun$. Then the following are equivalent:
\begin{enumerate}
\item
\label{p:2iiia}
For $\mu$-almost every $\omega\in\Omega$, $\TW$ is
$\alpha$-averaged.
\item
\label{p:2iiib}
$\dom T=\HH$ and $T$ is $\alpha$-averaged.
\end{enumerate}
\end{enumerate}
\end{proposition}
\begin{proof}
Observe that $T$ is at most single-valued.
On the other hand, Lemma~\ref{l:1}\ref{l:1v} states that
there exists a sequence $(u_n)_{n\in\NN}$ in $\GG$ such that
\begin{equation}
\label{e:a5ln}
(\forall\omega\in\Omega)\quad
\overline{\set!{u_n(\omega)}_{n\in\NN}}=\HW.
\end{equation}

\ref{p:2ia}$\Rightarrow$\ref{p:2ib}:
Let $\Xi\in\FF$ be such that $\mu(\Xi)=0$ and, for every
$\omega\in\complement\Xi$, $\TW$ is $\beta$-Lipschitzian. Then
\begin{equation}
\label{e:6ces}
(\forall x\in\GG)(\forall y\in\GG)
\brk!{\forall\omega\in\complement\Xi}\quad
\norm!{\TW\brk!{x(\omega)}-\TW\brk!{y(\omega)}}_{\HW}
\leq\beta\norm{x(\omega)-y(\omega)}_{\HW}.
\end{equation}
In turn, since $\mathfrak{G}$ is a vector subspace of
$\prod_{\omega\in\Omega}\HW$, we infer from \ref{p:2A} and
\eqref{e:4a5t} that, for every $x\in\GG$ and every $y\in\GG$, the
mapping $\omega\mapsto\TW(x(\omega))-\TW(y(\omega))$ lies in
$\GG$. Thus, \ref{p:2B} implies that, for every $x\in\GG$, the
mapping $\omega\mapsto\TW(x(\omega))$ lies in $\GG$ as the sum of
two mappings in $\GG$, namely
$\omega\mapsto\TW(x(\omega))-\TW(z(\omega))$ and
$\omega\mapsto\TW(z(\omega))$. Therefore $\dom T=\HH$.
Additionally, it results from \eqref{e:6ces} and \eqref{e:y9d2}
that $T$ is $\beta$-Lipschitzian.

\ref{p:2ib}$\Rightarrow$\ref{p:2ia}:
Fix temporarily $n\in\NN$ and $m\in\NN$.
For every $\Xi\in\FF$ such that $\mu(\Xi)<\pinf$,
since $1_{\Xi}u_n\in\GG$ and $1_{\Xi}u_m\in\GG$ thanks to
Lemma~\ref{l:1}\ref{l:1iii}, we derive from \eqref{e:y9d2} that
\begin{align}
\int_{\Xi}\norm!{
\TW\brk!{u_n(\omega)}-\TW\brk!{u_m(\omega)}}_{\HW}^2
\mu(d\omega)
&=\norm{T\brk{1_{\Xi}u_n}-T\brk{1_{\Xi}u_m}}_{\HH}^2
\nonumber\\
&\leq\beta^2\norm{1_{\Xi}u_n-1_{\Xi}u_m}_{\HH}^2
\nonumber\\
&=\int_{\Xi}\beta^2
\norm{u_n(\omega)-u_m(\omega)}_{\HW}^2\mu(d\omega).
\end{align}
Hence, since $(\Omega,\FF,\mu)$ is $\sigma$-finite,
there exists $\Xi_{n,m}\in\FF$ such that
\begin{equation}
\label{e:xqgj}
\mu\brk!{\Xi_{n,m}}=0\quad\text{and}\quad
\brk!{\forall\omega\in\complement\Xi_{n,m}}\;\;
\norm!{\TW\brk!{u_n(\omega)}-\TW\brk!{u_m(\omega)}}_{\HW}
\leq\beta\norm{u_n(\omega)-u_m(\omega)}_{\HW}.
\end{equation}
Now set $\Xi=\bigcup_{n\in\NN,m\in\NN}\Xi_{n,m}$, let
$\omega\in\complement\Xi$, let $\mathsf{x}\in\HW$,
and let $\mathsf{y}\in\HW$. Then, $\Xi\in\FF$ with $\mu(\Xi)=0$
and, in view of \eqref{e:a5ln}, there exist sequences
$(k_n)_{n\in\NN}$ and
$(l_n)_{n\in\NN}$ in $\NN$ such that
$u_{k_n}(\omega)\to\mathsf{x}$ and $u_{l_n}(\omega)\to\mathsf{y}$.
At the same time, by \eqref{e:xqgj},
\begin{equation}
(\forall n\in\NN)\quad\norm!{
\TW\brk!{u_{k_n}(\omega)}-\TW\brk!{u_{l_n}(\omega)}}_{\HW}
\leq\beta\norm{u_{k_n}(\omega)-u_{l_n}(\omega)}_{\HW}.
\end{equation}
Thus, since $\norm{\Cdot}_{\HW}$ is weakly lower semicontinuous,
letting $n\to\pinf$ and invoking the sequential
strong-to-weak continuity of
$\TW$, we get $\norm{\TW\mathsf{x}-\TW\mathsf{y}}_{\HW}
\leq\beta\norm{\mathsf{x}-\mathsf{y}}_{\HW}$.

\ref{p:2ii} and \ref{p:2iii}:
Argue as in \ref{p:2i}.
\end{proof}

\begin{proposition}
\label{p:16}
Suppose that Assumption~\ref{a:1} is in force and, for every
$\omega\in\Omega$, let $\AW\colon\HW\to 2^{\HW}$
be a set-valued operator. Set
\begin{equation}
A=\leftindex^{\mathfrak{G}}{\int}_{\Omega}^{\oplus}
\AW\mu(d\omega)
\end{equation}
and let $\gamma\in\RPP$. Then
\begin{equation}
J_{\gamma A}=\leftindex^{\mathfrak{G}}{\Int}_{\Omega}^{\oplus}
J_{\gamma\AW}\mu(d\omega)
\quad\text{and}\quad
\moyo{A}{\gamma}=
\leftindex^{\mathfrak{G}}{\Int}_{\Omega}^{\oplus}
\moyo{\AW}{\gamma}\,\mu(d\omega).
\end{equation}
\end{proposition}
\begin{proof}
Set $T=\leftindex^{\mathfrak{G}}{\int}_{\Omega}^{\oplus}
J_{\gamma\AW}\mu(d\omega)$. We derive from Definition~\ref{d:6}
and \cite[Proposition~23.2(ii)]{Livre1} that
\begin{align}
(\forall x\in\HH)\quad
Tx
&=\menge{p\in\HH}{(\forallmu\omega\in\Omega)\,\,
p(\omega)\in J_{\gamma\AW}\brk!{x(\omega)}}
\nonumber\\
&=\menge{p\in\HH}{(\forallmu\omega\in\Omega)\,\,
\gamma^{-1}\brk!{x(\omega)-p(\omega)}\in\AW\brk!{p(\omega)}}
\nonumber\\
&=\menge{p\in\HH}{\gamma^{-1}(x-p)\in Ap}
\nonumber\\
&=J_{\gamma A}x.
\end{align}
Likewise, upon setting
$R=\leftindex^{\mathfrak{G}}{\int}_{\Omega}^{\oplus}
\moyo{\AW}{\gamma}\,\mu(d\omega)$,
we deduce from
Definition~\ref{d:6} and
\cite[Proposition~23.2(iii)]{Livre1}
that
\begin{align}
(\forall x\in\HH)\quad
Rx
&=\menge{p\in\HH}{(\forallmu\omega\in\Omega)\,\,
p(\omega)\in\moyo{\AW}{\gamma}\brk!{x(\omega)}}
\nonumber\\
&=\menge{p\in\HH}{(\forallmu\omega\in\Omega)\,\,
p(\omega)\in\AW\brk!{x(\omega)-\gamma p(\omega)}}
\nonumber\\
&=\menge{p\in\HH}{p\in A(x-\gamma p)}
\nonumber\\
&=\moyo{A}{\gamma}x,
\end{align}
which completes the proof.
\end{proof}

\begin{assumption}
\label{a:2}
Assumption~\ref{a:1} and the following are in force:
\begin{enumerate}[label={\normalfont[\Alph*]}]
\item
\label{a:2a}
For every $\omega\in\Omega$, $\AW\colon\HW\to 2^{\HW}$
is maximally monotone.
\item
\label{a:2b}
For every $x\in\GG$, the mapping $\omega\mapsto
J_{\AW}(x(\omega))$ lies in $\mathfrak{G}$. 
\item
\label{a:2c}
$\dom\leftindex^{\mathfrak{G}}{\int}_{\Omega}^{\oplus}
\AW\mu(d\omega)\neq\emp$.
\end{enumerate}
\end{assumption}

\begin{proposition}
\label{p:6}
Suppose that Assumption~\ref{a:2} is in force.
Then the following hold:
\begin{enumerate}
\item
For every $\omega\in\Omega$, $\AW^{-1}\colon\HW\to 2^{\HW}$
is maximally monotone.
\item
For every $x\in\GG$, the mapping $\omega\mapsto
J_{\AW^{-1}}(x(\omega))$ lies in $\mathfrak{G}$. 
\item
$\dom\leftindex^{\mathfrak{G}}{\int}_{\Omega}^{\oplus}
\AW^{-1}\mu(d\omega)\neq\emp$.
\end{enumerate}
\end{proposition}
\begin{proof}
We infer from Assumption~\ref{a:2}\ref{a:2a} and
\cite[Propositions~20.22 and 23.20]{Livre1} that,
for every $\omega\in\Omega$, $\AW^{-1}$ is
maximally monotone and $J_{\AW^{-1}}=\Id_{\HW}-J_{\AW}$.
In turn, for every $x\in\GG$, since $\mathfrak{G}$ is a
vector subspace of $\prod_{\omega\in\Omega}\HW$,
it follows from Assumption~\ref{a:2}\ref{a:2b}
that the mapping $\omega\mapsto J_{\AW^{-1}}(x(\omega))$ lies in
$\mathfrak{G}$ as the difference of the mappings $x$ and
$\omega\mapsto J_{\AW}(x(\omega))$. Finally,
Proposition~\ref{p:1}\ref{p:1iv} and
Assumption~\ref{a:2}\ref{a:2c} yield
$\dom\leftindex^{\mathfrak{G}}{\int}_{\Omega}^{\oplus}
\AW^{-1}\mu(d\omega)=
\ran\leftindex^{\mathfrak{G}}{\int}_{\Omega}^{\oplus}
\AW\mu(d\omega)
\neq\emp$.
\end{proof}

The main result of this section is the following theorem, which
establishes the basic properties of Hilbert direct integrals of
maximally monotone operators. Special cases of items \ref{t:1i}
and \ref{t:1ii} corresponding to scenarios described in
Example~\ref{ex:1} can be found in
\cite{Atto79,Livre1,Brez73,Ibap21,Penn03}.

\begin{theorem}
\label{t:1}
Suppose that Assumption~\ref{a:2} is in force and set
\begin{equation}
\label{e:yrzn}
A=\leftindex^{\mathfrak{G}}{\int}_{\Omega}^{\oplus}
\AW\mu(d\omega).
\end{equation}
Then the following hold:
\begin{enumerate}
\item
\label{t:1i}
$A$ is maximally monotone.
\item
\label{t:1ii}
Let $\gamma\in\RPP$ and $x\in\GG$. Then the following are
satisfied:
\begin{enumerate}
\item
\label{t:1iia}
The mapping $\omega\mapsto J_{\gamma\AW}\brk{x(\omega)}$ lies in
$\GG$ and $J_{\gamma A}=
\leftindex^{\mathfrak{G}}{\Int}_{\Omega}^{\oplus}
J_{\gamma\AW}\mu(d\omega)$.
\item
\label{t:1iib}
The mapping $\omega\mapsto
\moyo{\AW}{\gamma}\brk{x(\omega)}$ lies in $\GG$
and $\moyo{A}{\gamma}=
\leftindex^{\mathfrak{G}}{\Int}_{\Omega}^{\oplus}
\moyo{\AW}{\gamma}\mu(d\omega)$.
\end{enumerate}
\item
\label{t:1iii}
$\cdom A=\leftindex^{\mathfrak{G}}{\Int}_{\Omega}^{\oplus}
\cdom\AW\,\mu(d\omega)
=\overline{\leftindex^{\mathfrak{G}}{\Int}_{\Omega}^{\oplus}
\dom\AW\,\mu(d\omega)}$.
\item
\label{t:1iv}
$\cran A=\leftindex^{\mathfrak{G}}{\Int}_{\Omega}^{\oplus}
\cran\AW\,\mu(d\omega)
=\overline{\leftindex^{\mathfrak{G}}{\Int}_{\Omega}^{\oplus}
\ran\AW\,\mu(d\omega)}$.
\item
\label{t:1v}
Let $x\in\GG$ be such that $(\forall\omega\in\Omega)$
$x(\omega)\in\dom\AW$. Then the following are satisfied:
\begin{enumerate}
\item
\label{t:1va}
The mapping $\omega\mapsto \moyo{\AW}{0}(x(\omega))$ lies in
$\mathfrak{G}$.
\item
\label{t:1vb}
Suppose that $x\in\dom A$. Then the mapping
$\omega\mapsto\moyo{\AW}{0}\brk{x(\omega)}$
lies in $\GG$ and $\moyo{A}{0}x=
\leftindex^{\mathfrak{G}}{\Int}_{\Omega}^{\oplus}
\moyo{\AW}{0}\brk!{x(\omega)}\mu(d\omega)$.
\end{enumerate}
\end{enumerate}
\end{theorem}
\begin{proof}
\ref{t:1i}:
By \cite[Proposition~23.2(i)]{Livre1}
and Assumption~\ref{a:2}\ref{a:2c}, $\ran J_A=\dom A\neq\emp$
and there thus exist $z$ and $r$ in $\HH$ such that $r\in J_Az$
or, equivalently, $z-r\in Ar$. Hence, for $\mu$-almost every
$\omega\in\Omega$, $z(\omega)-r(\omega)\in\AW(r(\omega))$
and, therefore, the monotonicity of $\AW$ yields
$r(\omega)=J_{\AW}(z(\omega))$. Thus, because $r\in\GG$,
we infer from Lemma~\ref{l:1}\ref{l:1ii} that the mapping
$\omega\mapsto J_{\AW}(z(\omega))$ lies in $\GG$.
In turn, appealing to Assumption~\ref{a:2}\ref{a:2b},
we deduce from Proposition~\ref{p:2}\ref{p:2iii}
(applied to the firmly nonexpansive operators
$(J_{\AW})_{\omega\in\Omega}$) and Proposition~\ref{p:16}
that $J_A\colon\HH\to\HH$ is firmly nonexpansive.
Consequently, \cite[Proposition~23.8(iii)]{Livre1}
guarantees that $A$ is maximally monotone.

\ref{t:1ii}:
Use \ref{t:1i}, Proposition~\ref{p:16}, and
Lemma~\ref{l:1}\ref{l:1ii}.

\ref{t:1iii}:
By \ref{t:1i} and \cite[Corollary~21.14]{Livre1},
$\cdom A$ is a nonempty closed convex subset of $\HH$.
Fix temporarily $x\in\GG$, let $(\gamma_n)_{n\in\NN}$ be a
sequence in $\zeroun$ such that $\gamma_n\downarrow 0$, and set
\begin{equation}
p=\proj_{\cdom A}x
\quad\text{and}\quad
(\forall n\in\NN)\;\;
p_n\colon\omega\mapsto J_{\gamma_n\AW}\brk!{x(\omega)}.
\end{equation}
We infer from \ref{t:1iia} that, for every $n\in\NN$, $p_n\in\GG$
and $p_n=J_{\gamma_nA}x$.
Thus, it follows from \ref{t:1i} and \cite[Theorem~23.48]{Livre1}
that $p_n\to p$ in $\HH$. In turn, Lemma~\ref{l:2} ensures that
there exist a strictly increasing sequence $(k_n)_{n\in\NN}$ in
$\NN$ and a set $\Xi\in\FF$ such that $\mu(\Xi)=0$ and
$(\forall\omega\in\complement\Xi)$ $p_{k_n}(\omega)\to p(\omega)$.
On the other hand, we deduce from Assumption~\ref{a:2}\ref{a:2a}
and \cite[Theorem~23.48]{Livre1} that
$(\forall\omega\in\complement\Xi)$
$p_{k_n}(\omega)=J_{\gamma_{k_n}\AW}(x(\omega))
\to\proj_{\cdom\AW}(x(\omega))$.
Therefore $(\forall\omega\in\complement\Xi)$
$p(\omega)=\proj_{\cdom\AW}(x(\omega))$.
Hence, because $p\in\GG$, it results from
Lemma~\ref{l:1}\ref{l:1ii} that the mapping
$\omega\mapsto\proj_{\cdom\AW}(x(\omega))$ is a representative in
$\GG$ of $\proj_{\cdom A}x$. This confirms that
\begin{equation}
\proj_{\cdom A}=
\leftindex^{\mathfrak{G}}{\Int}_{\Omega}^{\oplus}
\proj_{\cdom\AW}\mu(d\omega).
\end{equation}
Therefore, using Definition~\ref{d:3}, we get
\begin{align}
\cdom A
&=\menge{x\in\HH}{x=\proj_{\cdom A}x}
\nonumber\\
&=\Menge2{x\in\HH}{(\forallmu\omega\in\Omega)\,\,
x(\omega)=\proj_{\cdom\AW}\brk!{x(\omega)}}
\nonumber\\
&=\menge{x\in\HH}{(\forallmu\omega\in\Omega)\,\,
x(\omega)\in\cdom\AW}
\nonumber\\
&=\leftindex^{\mathfrak{G}}{\Int}_{\Omega}^{\oplus}
\cdom\AW\,\mu(d\omega).
\end{align}
Thus $\leftindex^{\mathfrak{G}}{\int}_{\Omega}^{\oplus}
\cdom\AW\,\mu(d\omega)$ is a closed subset of $\HH$.
Consequently, we deduce from Proposition~\ref{p:1}\ref{p:1i} and
Definition~\ref{d:3} that
\begin{equation}
\cdom A
\subset\overline{\leftindex^{\mathfrak{G}}{\Int}_{\Omega}^{\oplus}
\dom\AW\,\mu(d\omega)}
\subset\leftindex^{\mathfrak{G}}{\Int}_{\Omega}^{\oplus}
\cdom\AW\,\mu(d\omega)
=\cdom A,
\end{equation}
which furnishes the desired identities.

\ref{t:1iv}:
In the light of Proposition~\ref{p:1}\ref{p:1iv} and
Proposition~\ref{p:6}, the claim follows from
\ref{t:1iii} applied to the family 
$(\AW^{-1})_{\omega\in\Omega}$.

\ref{t:1v}:
Let $(\gamma_n)_{n\in\NN}$ be a sequence in $\zeroun$
such that $\gamma_n\downarrow 0$, and set
\begin{equation}
\label{e:j81i}
p\colon\omega\mapsto\moyo{\AW}{0}\brk!{x(\omega)}
\quad\text{and}\quad
(\forall n\in\NN)\;\;p_n\colon\omega\mapsto
\moyo{\AW}{\gamma_n}\brk!{x(\omega)}.
\end{equation}
Then, on account of \ref{t:1iib},
\begin{equation}
\label{e:xc5v}
(\forall n\in\NN)\quad p_n\in\GG
\quad\text{and}\quad
p_n=\moyo{A}{\gamma_n}x.
\end{equation}

\ref{t:1va}:
For every $\omega\in\Omega$, since $\AW$ is maximally monotone and
$x(\omega)\in\dom\AW$, \cite[Corollary~23.46(i)]{Livre1} yields
$p_n(\omega)\to p(\omega)$. Hence, thanks to
Lemma~\ref{l:1}\ref{l:1iv}, we obtain $p\in\mathfrak{G}$.

\ref{t:1vb}:
Set $q=\moyo{A}{0}x$. It follows from \eqref{e:xc5v}, \ref{t:1i},
and \cite[Corollary~23.46(i)]{Livre1} that
$p_n\to q$ in $\HH$. Thus, we infer from Lemma~\ref{l:2} that
there exists a strictly increasing sequence $(k_n)_{n\in\NN}$ in
$\NN$ such that $(\forallmu\omega\in\Omega)$
$p_{k_n}(\omega)\to q(\omega)$. In turn, $p=q$ $\mae$ and we
conclude by invoking Lemma~\ref{l:1}\ref{l:1ii}.
\end{proof}

\begin{example}
\label{ex:5}
Consider the setting of Example~\ref{ex:1}\ref{ex:1ii}
and suppose that, in addition, $(\forall k\in\NN)$
$\alpha_k=1$ and $\mathsf{H}_k=\RR$. Then $\HH=\ell^2(\NN)$.
Now define $(\forall k\in\NN)$
$\mathsf{A}_k\colon\mathsf{H}_k\to\mathsf{H}_k\colon
\mathsf{x}\mapsto 2^k\mathsf{x}$. Then
\begin{equation}
\dom\brk3{\leftindex^{\mathfrak{G}}{\int}_{\Omega}^{\oplus}
\AW\mu(d\omega)}
=\Menge4{(\mathsf{x}_k)_{k\in\NN}\in\ell^2(\NN)}{
\sum_{k\in\NN}4^k\abs{\mathsf{x}_k}^2<\pinf}
\neq\ell^2(\NN)
=\leftindex^{\mathfrak{G}}{\Int}_{\Omega}^{\oplus}
\dom\AW\,\mu(d\omega).
\end{equation}
The closure operation in items \ref{t:1iii} and \ref{t:1iv} in
Theorem~\ref{t:1} can therefore not be omitted.
\end{example}

\begin{corollary}
\label{c:6}
Let $(\Omega,\FF,\mu)$ be a complete $\sigma$-finite measure 
space, let $\mathsf{H}$ be a separable real Hilbert space,
and for every $\omega\in\Omega$, let
$\AW\colon\mathsf{H}\to 2^{\mathsf{H}}$ be maximally monotone.
Set $\HH=L^2(\Omega,\FF,\mu;\mathsf{H})$ and
\begin{equation}
\label{e:7d04}
A\colon\HH\to 2^{\HH}\colon x\mapsto
\menge{x^*\in\HH}{(\forallmu\omega\in\Omega)\,\,
x^*(\omega)\in\AW\brk!{x(\omega)}}.
\end{equation}
Suppose that $\dom A\neq\emp$.
Then the following are equivalent:
\begin{enumerate}
\item
\label{c:6i}
$A$ is maximally monotone.
\item
\label{c:6ii}
For every $\mathsf{x}\in\mathsf{H}$,
the mapping $\Omega\to\mathsf{H}\colon\omega\mapsto
J_{\AW}\mathsf{x}$ is $(\FF,\BE_{\mathsf{H}})$-measurable.
\item
\label{c:6iii}
For every open set $\boldsymbol{\mathsf{V}}$ in
$\mathsf{H}\oplus\mathsf{H}$,
$\menge{\omega\in\Omega}{\boldsymbol{\mathsf{V}}\cap
\gra\AW\neq\emp}\in\FF$.
\end{enumerate}
\end{corollary}
\begin{proof}
In the light of Example~\ref{ex:1}\ref{ex:1iii}, $\HH$ is the
Hilbert direct integral of the $\FF$-measurable vector field
$((\HW)_{\omega\in\Omega},\mathfrak{G})$ defined by
\begin{equation}
\label{e:fvor}
\brk[s]!{\;(\forall\omega\in\Omega)\;\;
\mathsf{H}_{\omega}=\mathsf{H}\;}
\quad\text{and}\quad
\mathfrak{G}=\menge{x\colon\Omega\to\mathsf{H}}{
x\,\,\text{is $(\FF,\BE_{\mathsf{H}})$-measurable}}.
\end{equation}
Additionally, by \eqref{e:7d04},
\begin{equation}
\label{e:lpza}
A=\leftindex^{\mathfrak{G}}{\int}_{\Omega}^{\oplus}
\AW\mu(d\omega).
\end{equation}

\ref{c:6i}$\Rightarrow$\ref{c:6ii}:
We have $\dom A\neq\emp$ and $J_A\colon\HH\to\HH$
\cite[Corollary~23.11(i)]{Livre1}.
Thus, invoking Proposition~\ref{p:16} and
Lemma~\ref{l:1}\ref{l:1ii}, we deduce that
\begin{equation}
\label{e:g0d4}
\brk!{\forall x\in\mathscr{L}^2\brk!{\Omega,\FF,\mu;\mathsf{H}}}
\quad\text{the mapping $\Omega\to\mathsf{H}\colon\omega\mapsto
J_{\AW}\brk!{x(\omega)}$ lies in $
\mathscr{L}^2\brk!{\Omega,\FF,\mu;\mathsf{H}}$}.
\end{equation}
Next, take $\mathsf{x}\in\mathsf{H}$.
Since $(\Omega,\FF,\mu)$ is $\sigma$-finite,
there exists an increasing sequence $(\Omega_n)_{n\in\NN}$ 
in $\FF$ of finite $\mu$-measure sets such that
$\bigcup_{n\in\NN}\Omega_n=\Omega$. In turn,
$\set{1_{\Omega_n}\mathsf{x}}_{n\in\NN}\subset
\mathscr{L}^2\brk{\Omega,\FF,\mu;\mathsf{H}}$ and
$(\forall\omega\in\Omega)$
$1_{\Omega_n}(\omega)\mathsf{x}\to\mathsf{x}$.
Hence, on account of \eqref{e:g0d4},
we deduce that, for every $n\in\NN$, the mapping
$\Omega\to\mathsf{H}\colon
\omega\mapsto J_{\AW}(1_{\Omega_n}(\omega)\mathsf{x})$ is
$(\FF,\BE_{\mathsf{H}})$-measurable. In addition,
the continuity of the operators $(J_{\AW})_{\omega\in\Omega}$
yields $(\forall\omega\in\Omega)$
$J_{\AW}(1_{\Omega_n}(\omega)\mathsf{x})\to J_{\AW}\mathsf{x}$.
Altogether, it results from Lemma~\ref{l:1}\ref{l:1iv}
that the mapping $\Omega\to\mathsf{H}\colon
\omega\mapsto J_{\AW}\mathsf{x}$ is
$(\FF,\BE_{\mathsf{H}})$-measurable.

\ref{c:6ii}$\Rightarrow$\ref{c:6i}:
Applying \cite[Lemma~III.14]{Cast77} to the mapping
$\Omega\times\mathsf{H}\to\mathsf{H}\colon
(\omega,\mathsf{x})\mapsto J_{\AW}\mathsf{x}$, we deduce that,
for every $x\in\mathfrak{G}$, the mapping $\omega\mapsto
J_{\AW}(x(\omega))$ lies in $\mathfrak{G}$.
Therefore, in the setting of \eqref{e:fvor}, the family
$(\AW)_{\omega\in\Omega}$ satisfies Assumption~\ref{a:2}.
Consequently, we conclude via \eqref{e:lpza} and
Theorem~\ref{t:1}\ref{t:1i} that $A$ is maximally monotone.

\ref{c:6ii}$\Leftrightarrow$\ref{c:6iii}:
Combine \cite[Lemme~2.1]{Atto79} and
\cite[Th\'eor\`eme~2.1]{Atto79}.
\end{proof}

\begin{remark}
The implication \ref{c:6iii}$\Rightarrow$\ref{c:6i} in
Corollary~\ref{c:6} is stated in \cite[Theorem~5.1]{Penn03}.
\end{remark}

\begin{proposition}
\label{p:9}
Suppose that Assumption~\ref{a:1} is in force.
Let $\mathsf{G}$ be a separable real Hilbert space and,
for every $\omega\in\Omega$, let
$\LW\colon\mathsf{G}\to\HW$ be linear and bounded.
Suppose that, for every $\mathsf{z}\in\mathsf{G}$, the mapping
\begin{equation}
\label{e:qk0f}
\mathfrak{e}_{\mathsf{L}}\mathsf{z}\colon
\omega\mapsto\LW\mathsf{z}
\end{equation}
lies in $\mathfrak{G}$. Then the following holds:
\begin{enumerate}
\item
\label{p:9i-}
The function $\Omega\to\RR\colon\omega\mapsto\norm{\LW}$ is
$\FF$-measurable.
\end{enumerate}
Suppose additionally that
$\int_{\Omega}\norm{\LW}^2\mu(d\omega)<\pinf$ and define
\begin{equation}
\label{e:n9we}
L\colon\mathsf{G}\to\HH\colon\mathsf{z}\mapsto
\mathfrak{e}_{\mathsf{L}}\mathsf{z}.
\end{equation}
Then the following hold:
\begin{enumerate}[resume]
\item
\label{p:9i}
$L$ is well defined, linear, and bounded with $\norm{L}\leq
\sqrt{\int_{\Omega}\norm{\LW}^2\mu(d\omega)}$.
\item
\label{p:9ii}
Let $x^*\in\mathfrak{G}$. Then the mapping
$\Omega\to\mathsf{G}\colon\omega\mapsto\LW^*(x^*(\omega))$ is
$(\FF,\BE_{\mathsf{G}})$-measurable.
\item
\label{p:9iii}
Let $x^*\in\GG$. Then the mapping
$\Omega\to\mathsf{G}\colon\omega\mapsto\LW^*(x^*(\omega))$
is Lebesgue $\mu$-integrable.
\item
\label{p:9iv}
$L^*\colon\HH\to\mathsf{G}\colon x^*\mapsto
\int_{\Omega}\LW^*(x^*(\omega))\mu(d\omega)$.
\end{enumerate}
\end{proposition}
\begin{proof}
\ref{p:9i-}:
Let $\set{\mathsf{z}_n}_{n\in\NN}$ be a dense subset of
$\menge{\mathsf{z}\in\mathsf{G}}{
\norm{\mathsf{z}}_{\mathsf{G}}\leq 1}$. 
On the one hand, property~\ref{d:1a} in Definition~\ref{d:1}
ensures that, for every $n\in\NN$, the function
$\Omega\to\RR\colon\omega\mapsto\norm{\LW\mathsf{z}_n}_{\HW}$
is $\FF$-measurable. On the other hand, thanks to the continuity
of the operators $(\LW)_{\omega\in\Omega}$,
\begin{equation}
(\forall\omega\in\Omega)\quad
\norm{\LW}
=\sup_{\substack{\mathsf{z}\in\mathsf{G}\\
\norm{\mathsf{z}}_{\mathsf{G}}\leq 1}}
\norm{\LW\mathsf{z}}_{\HW}
=\sup_{n\in\NN}
\norm{\LW\mathsf{z}_n}_{\HW}.
\end{equation}
Altogether, the function $\Omega\to\RR\colon\omega\mapsto
\norm{\LW}$ is $\FF$-measurable.

\ref{p:9i}:
For every $\mathsf{z}\in\mathsf{G}$, we deduce from
\eqref{e:qk0f} that
\begin{equation}
\label{e:692z}
\int_{\Omega}
\norm{(\mathfrak{e}_{\mathsf{L}}\mathsf{z})(\omega)}_{\HW}^2
\mu(d\omega)
=\int_{\Omega}\norm{\LW\mathsf{z}}_{\HW}^2\mu(d\omega)
\leq\norm{\mathsf{z}}_{\mathsf{G}}^2
\int_{\Omega}\norm{\LW}^2
\mu(d\omega)
<\pinf
\end{equation}
and, in turn, from \eqref{e:4a5t} that
$\mathfrak{e}_{\mathsf{L}}\mathsf{z}\in\GG$.
This confirms that $L$ is well defined. In addition, the linearity
of the operators $(\LW)_{\omega\in\Omega}$
guarantees that of $L$. The last claims follow from
\eqref{e:692z} and \eqref{e:y9d2}.

\ref{p:9ii}:
For every $\mathsf{z}\in\mathsf{G}$, Lemma~\ref{l:1}\ref{l:1i}
implies that the function $\Omega\to\RR\colon\omega\mapsto
\scal{\mathsf{z}}{\LW^*(x^*(\omega))}_{\mathsf{G}}
=\scal{\LW\mathsf{z}}{x^*(\omega)}_{\HW}$ is $\FF$-measurable.
In turn, invoking the separability of
$\mathsf{G}$, as well as the fact that $(\Omega,\FF,\mu)$ is
complete and $\sigma$-finite, we derive from
\cite[Th\'eor\`eme~5.6.24]{Sch93b} that the mapping
$\Omega\to\mathsf{G}\colon\omega\mapsto\LW^*(x^*(\omega))$ is
$(\FF,\BE_{\mathsf{G}})$-measurable.

\ref{p:9iii}:
By the Cauchy--Schwarz inequality, 
\begin{align}
\int_{\Omega}\norm!{\LW^*\brk!{x^*(\omega)}}_{\mathsf{G}}\,
\mu(d\omega)
&\leq\int_{\Omega}\norm{\LW}\,\norm{x^*(\omega)}_{\HW}\mu(d\omega)
\nonumber\\
&\leq\sqrt{\int_{\Omega}\norm{\LW}^2\mu(d\omega)}
\sqrt{\int_{\Omega}\norm{x^*(\omega)}_{\HW}^2\mu(d\omega)}
\nonumber\\
&<\pinf.
\end{align}
Hence, the assertion follows from
\cite[Th\'eor\`eme~5.7.21]{Sch93b}.

\ref{p:9iv}:
Take $x^*\in\HH$. It results from \eqref{e:y9d2},
\eqref{e:n9we}, \eqref{e:qk0f}, \ref{p:9iii}, and
\cite[Th\'eor\`eme~5.8.16]{Sch93b} that
\begin{align}
(\forall\mathsf{z}\in\mathsf{G})\quad
\scal{\mathsf{z}}{L^*x^*}_{\mathsf{G}}
&=\scal{L\mathsf{z}}{x^*}_{\HH}
\nonumber\\
&=\int_{\Omega}\scal{\LW\mathsf{z}}{x^*(\omega)}_{
\HW}\mu(d\omega)
\nonumber\\
&=\int_{\Omega}\scal!{\mathsf{z}}{\LW^*
\brk!{x^*(\omega)}}_{\mathsf{G}}\,\mu(d\omega)
\nonumber\\
&=\scal3{\mathsf{z}}{\int_{\Omega}
\LW^*\brk!{x^*(\omega)}\mu(d\omega)}_{\mathsf{G}},
\end{align}
which completes the proof.
\end{proof}

\section{Hilbert direct integrals of functions}
\label{sec:4}

We study the Hilbert direct integrals of families of functions
introduced in Definition~\ref{d:5}.

\begin{lemma}
\label{l:3}
Let $\HH$ be a real Hilbert space and let $T\colon\HH\to\HH$.
Then the following hold:
\begin{enumerate}
\item
\label{l:3i}
There exists $f\in\Gamma_0(\HH)$ such that $T=\prox_f$ if and only
if $T$ is nonexpansive and cyclically monotone, that is,
for every $2\leq n\in\NN$ and every
$(x_1,\ldots,x_{n+1})\in\HH^{n+1}$ such that $x_{n+1}=x_1$,
\begin{equation}
\sum_{k=1}^n\scal{x_{k+1}-x_k}{Tx_k}_{\HH}\leq 0.
\end{equation}
\item
\label{l:3ii}
There exists a nonempty closed convex subset $C$ of $\HH$ such
that $T=\proj_C$ if and only if
\begin{equation}
(\forall x\in\HH)(\forall y\in\HH)\quad
\scal{Ty-Tx}{x-Tx}_{\HH}\leq 0.
\end{equation}
\end{enumerate}
\end{lemma}
\begin{proof}
\ref{l:3i}:
The core of our argument is implicitly in
\cite[Corollaire~10.c]{More65}. Suppose that there exists
$f\in\Gamma_0(\HH)$ such that $T=\prox_f$. Then, on account of
\cite[Corollaire~10.c]{More65} and
\cite[Proposition~22.14]{Livre1}, $T$ is nonexpansive and
cyclically monotone. Conversely, suppose that $T$ is nonexpansive
and cyclically monotone. Then $T$ is monotone and it thus follows
from \cite[Corollary~20.28]{Livre1} that $T$ is maximally
monotone. Therefore, Rockafellar's cyclic monotonicity theorem
\cite[Theorem~22.18]{Livre1} guarantees the existence of a
function $\varphi\in\Gamma_0(\HH)$ such that $T=\partial\varphi$.
We conclude by invoking \cite[Corollaire~10.c]{More65}.

\ref{l:3ii}:
See \cite[Theorem~1.1]{Zara71}.
\end{proof}

\begin{remark}
\label{r:9}
In connection with Lemma~\ref{l:3}\ref{l:3i}, a characterization of
proximity operators based on firm nonexpansiveness and an
alternative cyclic inequality is provided in
\cite[Theorem~6.6]{Bart07}.
\end{remark}

In \cite{Mor63a,More65}, Moreau showed that the convex combination
of finitely many proximity operators acting on the same Hilbert
space is a proximity operator. Here is a generalization of this
result.

\begin{theorem}
\label{t:41}
Suppose that Assumption~\ref{a:1} is in force.
Let $\mathsf{G}$ be a separable real Hilbert space and,
for every $\omega\in\Omega$, let
$\mathsf{f}_{\omega}\in\Gamma_0(\HW)$ and let
$\LW\colon\mathsf{G}\to\HW$ be linear and bounded.
Suppose that the following are satisfied:
\begin{enumerate}[label={\normalfont[\Alph*]}]
\item
For every $x\in\GG$, the mapping
$\omega\mapsto\prox_{\mathsf{f}_{\omega}}(x(\omega))$ lies
in $\mathfrak{G}$.
\item
There exists $z\in\GG$ such that the mapping
$\omega\mapsto\prox_{\mathsf{f}_{\omega}}(z(\omega))$ lies in
$\GG$. 
\item
For every $\mathsf{z}\in\mathsf{G}$, the mapping
$\mathfrak{e}_{\mathsf{L}}\mathsf{z}\colon
\omega\mapsto\LW\mathsf{z}$ lies in $\mathfrak{G}$.
\item
$\int_{\Omega}\norm{\LW}^2\mu(d\omega)\leq 1$.
\end{enumerate}
Then
\begin{equation}
\label{e:i5}
\brk!{\exi\mathsf{g}\in\Gamma_0(\mathsf{G})}
(\forall\mathsf{z}\in\mathsf{G})\quad
\prox_{\mathsf{g}}\mathsf{z}=\int_{\Omega}
\LW^*\brk!{\prox_{\mathsf{f}_{\omega}}(\LW\mathsf{z})}
\mu(d\omega).
\end{equation}
\end{theorem}
\begin{proof}
Set $T=\leftindex^{\mathfrak{G}}{\int}_{\Omega}^{\oplus}
\prox_{\mathsf{f}_{\omega}}\mu(d\omega)$.
Then, on account of Proposition~\ref{p:2}\ref{p:2i},
$T\colon\HH\to\HH$ is nonexpansive. Next, items \ref{p:9i} and
\ref{p:9iv} of Proposition~\ref{p:9} ensure that the operator
$L\colon\mathsf{G}\to\HH\colon\mathsf{z}\mapsto
\mathfrak{e}_{\mathsf{L}}\mathsf{z}$ is well defined,
linear, and bounded, with $\norm{L}\leq 1$, and its adjoint
is given by
\begin{equation}
L^*\colon\HH\to\mathsf{G}\colon x^*\mapsto
\int_{\Omega}\LW^*\brk!{x^*(\omega)}\mu(d\omega).
\end{equation}
Hence, $L^*\circ T\circ L\colon\mathsf{G}\to\mathsf{G}$
is nonexpansive and
\begin{equation}
(\forall\mathsf{z}\in\mathsf{G})\quad
L^*\brk!{T(L\mathsf{z})}=\int_{\Omega}
\LW^*\brk!{\prox_{\mathsf{f}_{\omega}}(\LW\mathsf{z})}
\mu(d\omega).
\end{equation}
Therefore, in the light of Lemma~\ref{l:3}\ref{l:3i}, it remains
to show that $L^*\circ T\circ L$ is cyclically monotone.
Towards this end, let $2\leq n\in\NN$ and let
$(\mathsf{z}_1,\ldots,\mathsf{z}_{n+1})\in\mathsf{G}^{n+1}$
be such that $\mathsf{z}_{n+1}=\mathsf{z}_1$.
Then, appealing to the cyclic monotonicity of the operators
$(\prox_{\mathsf{f}_{\omega}})_{\omega\in\Omega}$,
\begin{equation}
(\forall\omega\in\Omega)\quad
\sum_{k=1}^n\scal!{\LW\mathsf{z}_{k+1}-\LW\mathsf{z}_k}{
\prox_{\mathsf{f}_{\omega}}(\LW\mathsf{z}_k)}_{\HW}\leq 0.
\end{equation}
Thus, it follows from \eqref{e:y9d2} that
\begin{align}
\sum_{k=1}^n\scal!{\mathsf{z}_{k+1}-\mathsf{z}_k}{
L^*\brk!{T(L\mathsf{z}_k)}}_{\mathsf{G}}
&=\sum_{k=1}^n\scal{L\mathsf{z}_{k+1}-L\mathsf{z}_k}{
T(L\mathsf{z}_k)}_{\HH}
\nonumber\\
&=\sum_{k=1}^n\int_{\Omega}
\scal!{\LW\mathsf{z}_{k+1}-\LW\mathsf{z}_k}{
\prox_{\mathsf{f}_{\omega}}(\LW\mathsf{z}_k)}_{\HW}\mu(d\omega)
\nonumber\\
&=\int_{\Omega}\sum_{k=1}^n
\scal!{\LW\mathsf{z}_{k+1}-\LW\mathsf{z}_k}{
\prox_{\mathsf{f}_{\omega}}(\LW\mathsf{z}_k)}_{\HW}\mu(d\omega)
\nonumber\\
&\leq 0,
\end{align}
which concludes the proof.
\end{proof}

\begin{remark}
\label{r:41}
Identifying the function $\mathsf{g}$ in \eqref{e:i5} is a natural 
question, which led to the introduction of the notion of integral 
proximal mixtures in \cite{Jota24}.
\end{remark}

\begin{proposition}
\label{p:18}
Suppose that Assumption~\ref{a:1} is in force and, for every
$\omega\in\Omega$, let
$\AW\colon\HW\to 2^{\HW}$ be maximally monotone.
Set
\begin{equation}
A=\leftindex^{\mathfrak{G}}{\int}_{\Omega}^{\oplus}
\AW\mu(d\omega).
\end{equation}
Then the following hold:
\begin{enumerate}
\item
\label{p:18i}
Suppose that there exists $f\in\Gamma_0(\HH)$ such that
$A=\partial f$. Then, for $\mu$-almost every $\omega\in\Omega$,
there exists $\mathsf{f}_{\omega}\in\Gamma_0(\HW)$ such that
$\AW=\partial\mathsf{f}_{\omega}$.
\item
\label{p:18ii}
Suppose that there exists a nonempty closed convex subset
$C$ of $\HH$ such that $A=N_C$. Then, for $\mu$-almost every
$\omega\in\Omega$, there exists a nonempty closed convex subset
$\mathsf{C}_{\omega}$ of $\HW$ such that
$\AW=N_{\mathsf{C}_{\omega}}$.
\end{enumerate}
\end{proposition}
\begin{proof}
Lemma~\ref{l:1}\ref{l:1v} asserts that there exists
a sequence $(u_n)_{n\in\NN}$ in $\GG$ such that
\begin{equation}
\label{e:5umr}
(\forall\omega\in\Omega)\quad
\overline{\set!{u_n(\omega)}_{n\in\NN}}=\HW.
\end{equation}

\ref{p:18i}:
Set $\mathbb{I}=\menge{(i_k)_{1\leq k\leq n+1}\in\NN^{n+1}}{
2\leq n\in\NN\,\,\text{and}\,\,i_{n+1}=i_1}$, fix temporarily
$\boldsymbol{\mathrm{i}}=(i_k)_{1\leq k\leq n+1}\in\mathbb{I}$,
and let $\Theta\in\FF$ be such that $\mu(\Theta)<\pinf$.
Then, by Lemma~\ref{l:1}\ref{l:1iii},
$\set{1_{\Theta}u_{i_k}}_{1\leq k\leq n}\subset\GG$.
In turn, since $J_A\colon\HH\to\HH$, it follows from
Proposition~\ref{p:16} and Lemma~\ref{l:1}\ref{l:1ii} that,
for every $k\in\set{1,\ldots,n}$, a representative in
$\GG$ of $J_A(1_{\Theta}u_{i_k})$ is the mapping
\begin{equation}
\omega\mapsto
\begin{cases}
J_{\AW}\brk!{u_{i_k}(\omega)},
&\text{if}\,\,\omega\in\Theta;\\
J_{\AW}\mathsf{0},&\text{if}\,\,\omega\in\complement\Theta.
\end{cases}
\end{equation}
At the same time, for every $k\in\set{1,\ldots,n}$,
a representative in $\GG$ of $1_{\Theta}u_{i_k}$ is the mapping
\begin{equation}
\omega\mapsto
\begin{cases}
u_{i_k}(\omega),&\text{if}\,\,\omega\in\Theta;\\
\mathsf{0},&\text{if}\,\,\omega\in\complement\Theta.
\end{cases}
\end{equation}
Hence, since $J_A=\prox_f$ is cyclically monotone by virtue of
\cite[Example~23.3]{Livre1} and Lemma~\ref{l:3}, we derive
from \eqref{e:y9d2} that
\begin{equation}
\int_{\Theta}\sum_{k=1}^n\scal!{u_{i_{k+1}}(\omega)-
u_{i_k}(\omega)}{J_{\AW}\brk!{u_{i_k}(\omega)}}_{\HW}\mu(d\omega)
=\sum_{k=1}^n\scal{1_{\Theta}u_{i_{k+1}}-1_{\Theta}u_{i_k}}{
J_A(1_{\Theta}u_{i_k})}_{\HH}
\leq 0.
\end{equation}
Therefore, thanks to the fact that
$(\Omega,\FF,\mu)$ is $\sigma$-finite, there exists
$\Xi_{\boldsymbol{\mathrm{i}}}\in\FF$ such that
\begin{equation}
\label{e:6ewx}
\mu(\Xi_{\boldsymbol{\mathrm{i}}})=0
\quad\text{and}\quad
\brk!{\forall\omega\in\complement\Xi_{\boldsymbol{\mathrm{i}}}}
\,\,
\sum_{k=1}^n\scal!{u_{i_{k+1}}(\omega)-
u_{i_k}(\omega)}{J_{\AW}\brk!{u_{i_k}(\omega)}}_{\HW}
\leq 0.
\end{equation}
Now set $\Xi=\bigcup_{\boldsymbol{\mathrm{i}}\in\mathbb{I}}
\Xi_{\boldsymbol{\mathrm{i}}}$.
Since $\mathbb{I}$ is countable, $\Xi\in\FF$ and $\mu(\Xi)=0$.
Additionally, \eqref{e:6ewx} implies that
\begin{equation}
\label{e:kftj}
\brk!{\forall\boldsymbol{\mathrm{i}}=(i_k)_{1\leq k\leq n+1}
\in\mathbb{I}}
\brk!{\forall\omega\in\complement\Xi}
\,\,
\sum_{k=1}^n\scal!{u_{i_{k+1}}(\omega)-
u_{i_k}(\omega)}{J_{\AW}\brk!{u_{i_k}(\omega)}}_{\HW}
\leq 0.
\end{equation}
To proceed further, take $\omega\in\complement\Xi$,
let $2\leq n\in\NN$, and let
$(\mathsf{x}_1,\ldots,\mathsf{x}_{n+1})$ be a family
in $\HW$ such that $\mathsf{x}_{n+1}=\mathsf{x}_1$.
For every $k\in\set{1,\ldots,n}$, we infer from \eqref{e:5umr}
that there exists a sequence
$(i_{k,m})_{m\in\NN}$ in $\NN$ such that
$u_{i_{k,m}}(\omega)\to\mathsf{x}_k$.
Set $(\forall m\in\NN)$ $i_{n+1,m}=i_{1,m}$. Then, for every
$m\in\NN$, because $(i_{k,m})_{1\leq k\leq n+1}\in
\mathbb{I}$, it results from \eqref{e:kftj} that
\begin{equation}
\sum_{k=1}^n\scal!{u_{i_{k+1,m}}(\omega)-
u_{i_{k,m}}(\omega)}{
J_{\AW}\brk!{u_{i_{k,m}}(\omega)}}_{\HW}
\leq 0.
\end{equation}
Therefore, the continuity of $J_{\AW}$ forces
$\sum_{k=1}^n\scal{\mathsf{x}_{k+1}-\mathsf{x}_k}{
J_{\AW}\mathsf{x}_k}_{\HW}\leq 0$. Consequently,
since $J_{\AW}$ is nonexpansive, we conclude via
Lemma~\ref{l:3}\ref{l:3i} that there exists
$\mathsf{f}_{\omega}\in\Gamma_0(\HW)$ such that
$J_{\AW}=\prox_{\mathsf{f}_{\omega}}$.

\ref{p:18ii}:
Argue as in \ref{p:18i}.
\end{proof}

Let us collect the main properties of Hilbert direct integral
functions under the umbrella of the following assumption.

\begin{assumption}
\label{a:3}
Assumption~\ref{a:1} and the following are in force:
\begin{enumerate}[label={\normalfont[\Alph*]}]
\item
\label{a:3a}
For every $\omega\in\Omega$,
$\mathsf{f}_{\omega}\in\Gamma_0(\HW)$.
\item
\label{a:3b}
For every $x\in\GG$, the mapping
$\omega\mapsto\prox_{\mathsf{f}_{\omega}}(x(\omega))$ lies
in $\mathfrak{G}$.
\item
\label{a:3c}
There exists $r\in\GG$ such that the function
$\omega\mapsto\mathsf{f}_{\omega}(r(\omega))$
lies in $\mathscr{L}^1(\Omega,\FF,\mu;\RR)$.
\item
\label{a:3d}
There exist $s^*\in\GG$ and
$\vartheta\in\mathscr{L}^1(\Omega,\FF,\mu;\RR)$ such that
\begin{equation}
\label{e:30v2}
(\forallmu\omega\in\Omega)\quad\mathsf{f}_{\omega}\geq
\scal{\Cdot}{s^*(\omega)}_{\HW}+\vartheta(\omega).
\end{equation}
\end{enumerate}
\end{assumption}

The following theorem presents the main properties of Hilbert
direct integrals of convex functions. In the literature, such
properties are available only in the setting of
Examples~\ref{ex:1}\ref{ex:1i+} and \ref{ex:1}\ref{ex:1iii}; see
\cite{Livre1,Inte23,Cast77,Rock74}, where different techniques are
employed which are not applicable in our general context.

\begin{theorem}
\label{t:2}
Suppose that Assumption~\ref{a:3} is in force and define
\begin{equation}
\label{e:athb}
f=\leftindex^{\mathfrak{G}}{\int}_{\Omega}^{\oplus}
\mathsf{f}_{\omega}\mu(d\omega).
\end{equation}
Then the following hold:
\begin{enumerate}
\item
\label{t:2i-}
$f$ is well defined.
\item
\label{t:2i}
$f\in\Gamma_0(\HH)$.
\item
\label{t:2ii}
$\partial f=\leftindex^{\mathfrak{G}}{\Int}_{\Omega}^{\oplus}
\partial\mathsf{f}_{\omega}\mu(d\omega)$.
\item
\label{t:2iii}
Let $\gamma\in\RPP$ and $x\in\GG$. Then the mapping $\omega\mapsto
\prox_{\gamma\mathsf{f}_{\omega}}(x(\omega))$ lies in $\GG$
and $\prox_{\gamma f}=
\leftindex^{\mathfrak{G}}{\Int}_{\Omega}^{\oplus}
\prox_{\gamma\mathsf{f}_{\omega}}\mu(d\omega)$.
\item
\label{t:2iv}
$\cdom f=\leftindex^{\mathfrak{G}}{\Int}_{\Omega}^{\oplus}
\cdom\mathsf{f}_{\omega}\,\mu(d\omega)
=\overline{\leftindex^{\mathfrak{G}}{\Int}_{\Omega}^{\oplus}
\dom\mathsf{f}_{\omega}\,\mu(d\omega)}$.
\item
\label{t:2v}
$\Argmin f=\leftindex^{\mathfrak{G}}{\Int}_{\Omega}^{\oplus}
\Argmin\mathsf{f}_{\omega}\,\mu(d\omega)$.
\item
\label{t:2vi}
Let $\beta\in\RP$ and suppose that, for every $\omega\in\Omega$,
$\dom\mathsf{f}_{\omega}=\HW$ and
$\mathsf{f}_{\omega}$ is G\^ateaux differentiable on
$\HW$. Then the following are equivalent:
\begin{enumerate}
\item
For $\mu$-almost every $\omega\in\Omega$,
$\nabla\mathsf{f}_{\omega}$ is $\beta$-Lipschitzian.
\item
$\dom f=\HH$, $f$ is Fr\'echet differentiable, and $\nabla f$ is
$\beta$-Lipschitzian.
\end{enumerate}
\item
\label{t:2vii}
Let $\gamma\in\RPP$. Then $\moyo{f}{\gamma}=
\leftindex^{\mathfrak{G}}{\Int}_{\Omega}^{\oplus}
\moyo{\mathsf{f}_{\omega}}{\gamma}\mu(d\omega)$.
\item
\label{t:2viii}
$f^*=\leftindex^{\mathfrak{G}}{\Int}_{\Omega}^{\oplus}
\mathsf{f}_{\omega}^*\mu(d\omega)$.
\item
\label{t:2ix}
$\rec f=\leftindex^{\mathfrak{G}}{\Int}_{\Omega}^{\oplus}
\rec\mathsf{f}_{\omega}\,\mu(d\omega)$.
\end{enumerate}
\end{theorem}
\begin{proof}
According to \eqref{e:30v2}, there exists $\Xi\in\FF$ such that
\begin{equation}
\label{e:z5n1}
\mu(\Xi)=0
\end{equation}
and
\begin{equation}
\label{e:z9gx}
(\forall x\in\mathfrak{G})
\brk!{\forall\omega\in\complement\Xi}\quad
\mathsf{f}_{\omega}\brk!{x(\omega)}
\geq\scal{x(\omega)}{s^*(\omega)}_{\HW}+\vartheta(\omega).
\end{equation}
Let us define
\begin{equation}
p\colon\omega\mapsto
\prox_{\mathsf{f}_{\omega}}\brk!{r(\omega)+s^*(\omega)}.
\end{equation}
Since $r+s^*\in\GG$, Assumption~\ref{a:3}\ref{a:3b} ensures that
$p\in\mathfrak{G}$. In addition, we deduce from
\cite[Proposition~16.44]{Livre1} that
\begin{equation}
\label{e:fnu5}
\brk{\forall\omega\in\Omega}\quad
r(\omega)+s^*(\omega)-p(\omega)\in
\partial\mathsf{f}_{\omega}\brk!{p(\omega)}
\end{equation}
and, in turn, from \eqref{e:subdiff} and \eqref{e:z9gx} that
\begin{align}
\label{e:awml}
&\brk!{\forall\omega\in\complement\Xi}\quad
\mathsf{f}_{\omega}\brk!{r(\omega)}
-\scal{r(\omega)}{s^*(\omega)}_{\HW}
\nonumber\\
&\hskip 26mm
\geq\mathsf{f}_{\omega}\brk!{p(\omega)}+
\scal!{r(\omega)-p(\omega)}{
r(\omega)+s^*(\omega)-p(\omega)}_{\HW}-
\scal{r(\omega)}{s^*(\omega)}_{\HW}
\nonumber\\
&\hskip 26mm
=\mathsf{f}_{\omega}\brk!{p(\omega)}-
\scal{p(\omega)}{s^*(\omega)}_{\HW}+
\norm{r(\omega)-p(\omega)}_{\HW}^2
\nonumber\\
&\hskip 26mm
\geq\vartheta(\omega)+\norm{r(\omega)-p(\omega)}_{\HW}^2.
\end{align}
On the other hand, thanks to items
\ref{a:3c} and \ref{a:3d} in Assumption~\ref{a:3},
the function $\omega\mapsto\mathsf{f}_{\omega}(r(\omega))
-\scal{r(\omega)}{s^*(\omega)}_{\HW}
-\vartheta(\omega)$ lies in $\mathscr{L}^1(\Omega,\FF,\mu;\RR)$.
Therefore, it results from \eqref{e:awml} that $r-p\in\GG$ and,
since $r\in\GG$ by Assumption~\ref{a:3}\ref{a:3c}, we get
\begin{equation}
\label{e:u5ha}
p\in\GG.
\end{equation}
Now set
\begin{equation}
A=\leftindex^{\mathfrak{G}}{\int}_{\Omega}^{\oplus}
\partial\mathsf{f}_{\omega}\mu(d\omega).
\end{equation}
Assumption~\ref{a:3}\ref{a:3a} and 
\cite[Proposition~12.b]{More65} imply that the operators
$(\partial\mathsf{f}_{\omega})_{\omega\in\Omega}$ are maximally
monotone. Moreover, since $r+s^*\in\GG$,
we infer from \eqref{e:fnu5} and \eqref{e:u5ha} that $p\in\dom A$.
Moreover, Assumption~\ref{a:3}\ref{a:3b} and
\cite[Example~23.3]{Livre1} guarantee that, for every $x\in\GG$,
the mapping $\omega\mapsto
J_{\partial\mathsf{f}_{\omega}}(x(\omega))$ lies in
$\mathfrak{G}$. Altogether,
\begin{equation}
\label{e:k6sn}
\text{the family $(\partial\mathsf{f}_{\omega})_{\omega\in\Omega}$
satisfies the assumption of Theorem~\ref{t:1}}.
\end{equation}
Hence, it follows from Theorem~\ref{t:1}\ref{t:1i} that
\begin{equation}
\label{e:x7v9}
\text{$A$ is maximally monotone}
\end{equation}
and from Theorem~\ref{t:1}\ref{t:1iia} and
\cite[Example~23.3]{Livre1} that
\begin{equation}
\label{e:grde}
\brk!{\forall\gamma\in\RPP}(\forall x\in\GG)\quad
\text{the mapping}\,\,
\omega\mapsto\prox_{\gamma\mathsf{f}_{\omega}}\brk!{x(\omega)}
\,\,\text{lies in}\,\,\GG.
\end{equation}

\ref{t:2i-}:
We must show that, for every $x\in\GG$,
the function $\Omega\to\RX\colon\omega\mapsto
\mathsf{f}_{\omega}(x(\omega))$ is $\FF$-measurable.
To do so, we employ a Moreau envelope approximation technique
inspired by \cite{Atto79}. Take $x\in\GG$. For every
$\gamma\in\RPP$, let $\Psi_{\gamma}$ be the mapping defined on
$\intv{0}{1}\times\Omega$ by
\begin{equation}
\brk!{\forall(t,\omega)\in\intv{0}{1}\times\Omega}\quad
\Psi_{\gamma}(t,\omega)=r(\omega)+t\brk!{x(\omega)-r(\omega)}
-\prox_{\gamma\mathsf{f}_{\omega}}\brk!{
r(\omega)+t\brk!{x(\omega)-r(\omega)}}
\end{equation}
and define
\begin{equation}
\label{e:cqya}
\phi_{\gamma}\colon\intv{0}{1}\times\Omega\to\RR\colon
(t,\omega)\mapsto
\scal!{x(\omega)-r(\omega)}{\Psi_{\gamma}(t,\omega)
}_{\HW}.
\end{equation}
Then, for every $\gamma\in\RPP$, the continuity of the mappings
$(\Psi_{\gamma}(\Cdot,\omega))_{\omega\in\Omega}$
ensures that the functions
$(\phi_{\gamma}(\Cdot,\omega))_{\omega\in\Omega}$ are continuous,
while \eqref{e:grde} and Lemma~\ref{l:1}\ref{l:1i} ensure that the
functions $(\phi_{\gamma}(t,\Cdot))_{t\in\intv{0}{1}}$
are $\FF$-measurable. Hence, the functions
$(\phi_{\gamma})_{\gamma\in\RPP}$ are
$\BE_{\intv{0}{1}}\otimes\FF$-measurable
\cite[Lemma~III.14]{Cast77}. In turn, invoking the fact that
$(\Omega,\FF,\mu)$ is $\sigma$-finite, we deduce that,
for every $\gamma\in\RPP$, the function $\Omega\to\RR\colon
\omega\mapsto\int_0^1\phi_{\gamma}(t,\omega)dt$
is $\FF$-measurable. Therefore, for every $\gamma\in\RPP$,
since \cite[Proposition~12.30]{Livre1} implies that
\begin{equation}
(\forall\omega\in\Omega)\quad
\moyo{\mathsf{f}_{\omega}}{\gamma}\brk!{x(\omega)}-
\moyo{\mathsf{f}_{\omega}}{\gamma}\brk!{r(\omega)}
=\gamma^{-1}\int_0^1\scal!{x(\omega)-r(\omega)}{
\Psi_{\gamma}(t,\omega)}_{\HW}dt
=\gamma^{-1}\int_0^1\phi_{\gamma}(t,\omega)dt,
\end{equation}
we infer that the function
$\Omega\to\RR\colon\omega\mapsto
\moyo{\mathsf{f}_{\omega}}{\gamma}(x(\omega))-
\moyo{\mathsf{f}_{\omega}}{\gamma}(r(\omega))$
is $\FF$-measurable. However,
\cite[Proposition~12.33(ii)]{Livre1}
and Assumption~\ref{a:3}\ref{a:3c} give
\begin{equation}
(\forall\omega\in\Omega)\quad
\mathsf{f}_{\omega}\brk!{x(\omega)}-
\mathsf{f}_{\omega}\brk!{r(\omega)}=
\lim_{\gamma\downarrow 0}\brk2{
\moyo{\mathsf{f}_{\omega}}{\gamma}\brk!{x(\omega)}-
\moyo{\mathsf{f}_{\omega}}{\gamma}\brk!{r(\omega)}}.
\end{equation}
Hence, the function
$\Omega\to\RX\colon\omega\mapsto
\mathsf{f}_{\omega}(x(\omega))-
\mathsf{f}_{\omega}(r(\omega))$ is $\FF$-measurable.
Consequently, invoking Assumption~\ref{a:3}\ref{a:3c} once more,
we conclude that the function $\Omega\to\RX\colon\omega\mapsto
\mathsf{f}_{\omega}(x(\omega))$ is $\FF$-measurable.

\ref{t:2i}:
By \eqref{e:z9gx}, \eqref{e:z5n1}, and
Assumption~\ref{a:3}\ref{a:3d},
\begin{equation}
\brk{\forall x\in\HH}\quad
f(x)
=\int_{\Omega}\mathsf{f}_{\omega}\brk!{x(\omega)}\mu(d\omega)
\geq\int_{\Omega}
\scal{x(\omega)}{s^*(\omega)}_{\HW}\mu(d\omega)+
\int_{\Omega}\vartheta(\omega)\mu(d\omega)
>\minf,
\end{equation}
which yields
\begin{equation}
\minf\notin f(\HH).
\end{equation}
At the same time, since the functions
$(\mathsf{f}_{\omega})_{\omega\in\Omega}$ are convex by
Assumption~\ref{a:3}\ref{a:3a}, so is $f$.
Moreover, Assumption~\ref{a:3}\ref{a:3c} implies that
$\dom f\neq\emp$. Therefore, it remains to show that $f$ is lower
semicontinuous. Take $\xi\in\RR$, let $(x_n)_{n\in\NN}$ be a
sequence in $\HH$, let $x\in\HH$, and suppose that
\begin{equation}
\label{e:yf8g}
\sup_{n\in\NN}f(x_n)\leq\xi\;
\quad\text{and}\quad
x_n\to x.
\end{equation}
Then Lemma~\ref{l:2} asserts that there exists a strictly
increasing sequence $(k_n)_{n\in\NN}$ in $\NN$ such that
$(\forallmu\omega\in\Omega)$ $x_{k_n}(\omega)\to x(\omega)$.
Let us define
\begin{equation}
(\forall n\in\NN)\quad
\varrho_n\colon\Omega\to\RX\colon
\omega\mapsto\mathsf{f}_{\omega}\brk!{x_{k_n}(\omega)}
-\scal{x_{k_n}(\omega)}{s^*(\omega)}_{\HW}.
\end{equation}
By \ref{t:2i-} and Lemma~\ref{l:1}\ref{l:1i}, the functions
$(\varrho_n)_{n\in\NN}$ are $\FF$-measurable. Additionally,
\begin{equation}
(\forall n\in\NN)\quad
\varrho_n\geq\vartheta\,\,\mae\quad\text{and}\quad
\int_{\Omega}\varrho_n(\omega)\mu(d\omega)=
f(x_{k_n})-\scal{x_{k_n}}{s^*}_{\HH},
\end{equation}
and, since the functions $(\mathsf{f}_{\omega})_{\omega\in\Omega}$
are lower semicontinuous, $(\forallmu\omega\in\Omega)$
$\mathsf{f}_{\omega}(x(\omega))-
\scal{x(\omega)}{s^*(\omega)}_{\HW}
\leq\varliminf\varrho_n(\omega)$.
Thus, we derive from Fatou's lemma and \eqref{e:yf8g} that
\begin{align}
f(x)-\scal{x}{s^*}_{\HH}
&=\int_{\Omega}\brk2{\mathsf{f}_{\omega}\brk!{x(\omega)}
-\scal{x(\omega)}{s^*(\omega)}_{\HW}}\mu(d\omega)
\nonumber\\
&\leq\int_{\Omega}\varliminf\varrho_n(\omega)\,\mu(d\omega)
\nonumber\\
&\leq\varliminf\int_{\Omega}
\varrho_n(\omega)\mu(d\omega)
\nonumber\\
&=\varliminf\brk!{f(x_{k_n})-\scal{x_{k_n}}{s^*}_{\HH}}
\nonumber\\
&\leq\xi-\scal{x}{s^*}_{\HH}.
\end{align}
Hence $f(x)\leq\xi$ and we conclude via \cite[Lemma~1.24]{Livre1}
that $f$ is lower semicontinuous.

\ref{t:2ii}:
Let $(x,x^*)\in\gra A$ and let $\Theta\in\FF$ be such that
$\mu(\Theta)=0$ and $(\forall\omega\in\complement\Theta)$
$x^*(\omega)\in\partial\mathsf{f}_{\omega}(x(\omega))$.
For every $y\in\HH$, thanks to the inequalities
\begin{equation}
\brk!{\forall\omega\in\complement\Theta}\quad
\scal!{y(\omega)-x(\omega)}{x^*(\omega)}_{\HW}
+\mathsf{f}_{\omega}\brk!{x(\omega)}\leq
\mathsf{f}_{\omega}\brk!{y(\omega)},
\end{equation}
we obtain $\scal{y-x}{x^*}_{\HH}+f(x)\leq f(y)$. Hence,
$(x,x^*)\in\gra\partial f$ and we thus have
$\gra A\subset\gra\partial f$.
Consequently, the monotonicity of $\partial f$ and
\eqref{e:x7v9} force $\partial f=A$.

\ref{t:2iii}:
Combine \ref{t:2i}, \cite[Example~23.3]{Livre1}, \ref{t:2ii},
\eqref{e:k6sn}, and Theorem~\ref{t:1}\ref{t:1iia}.

\ref{t:2iv}:
We derive from
\ref{t:2i},
\cite[Proposition~16.38]{Livre1},
\ref{t:2ii},
\eqref{e:k6sn}, and Theorem~\ref{t:1}\ref{t:1iii} that
\begin{equation}
\cdom f
=\cdom\partial f
=\leftindex^{\mathfrak{G}}{\Int}_{\Omega}^{\oplus}
\cdom\partial\mathsf{f}_{\omega}\,\mu(d\omega)
=\leftindex^{\mathfrak{G}}{\Int}_{\Omega}^{\oplus}
\cdom\mathsf{f}_{\omega}\,\mu(d\omega).
\end{equation}
This shows that
$\leftindex^{\mathfrak{G}}{\int}_{\Omega}^{\oplus}
\cdom\mathsf{f}_{\omega}\,\mu(d\omega)$ is a closed subset of
$\HH$. On the other hand, for every $x\in\dom f$,
it results from Definition~\ref{d:5} that,
for $\mu$-almost every $\omega\in\Omega$,
$x(\omega)\in\dom\mathsf{f}_{\omega}$ and, therefore, that
$x\in\leftindex^{\mathfrak{G}}{\int}_{\Omega}^{\oplus}
\dom\mathsf{f}_{\omega}\,\mu(d\omega)$. Consequently,
\begin{equation}
\cdom f
\subset\overline{\leftindex^{\mathfrak{G}}{\int}_{\Omega}^{\oplus}
\dom\mathsf{f}_{\omega}\,\mu(d\omega)}
\subset
\leftindex^{\mathfrak{G}}{\Int}_{\Omega}^{\oplus}
\cdom\mathsf{f}_{\omega}\,\mu(d\omega)
=\cdom f,
\end{equation}
which yields the desired identities.

\ref{t:2v}:
This follows from Fermat's rule, \ref{t:2ii}, and
Proposition~\ref{p:1}\ref{p:1iii}.

\ref{t:2vi}:
Appealing to \eqref{e:k6sn}, we deduce from
Theorem~\ref{t:1}\ref{t:1va} and
\cite[Proposition~17.31(i)]{Livre1} that, for every $x\in\GG$,
the mapping $\omega\mapsto\moyo{(\partial\mathsf{f}_{\omega})}{0}
(x(\omega))=\nabla\mathsf{f}_{\omega}(x(\omega))$ lies in
$\mathfrak{G}$. In addition, by \ref{t:2ii},
\begin{equation}
\partial f=\leftindex^{\mathfrak{G}}{\Int}_{\Omega}^{\oplus}
\nabla\mathsf{f}_{\omega}\,\mu(d\omega).
\end{equation}
Furthermore, for every $\omega\in\Omega$,
\cite[Corollary~17.40]{Livre1} asserts that
$\nabla\mathsf{f}_{\omega}\colon\HW\to\HW$ is strong-to-weak
continuous. Consequently, in the light of
\cite[Proposition~17.41]{Livre1},
the assertion follows from Proposition~\ref{p:2}\ref{p:2i} applied
to the operators $(\nabla\mathsf{f}_{\omega})_{\omega\in\Omega}$.

\ref{t:2vii}:
Take $x\in\GG$ and define $q\colon\omega\mapsto
\prox_{\gamma\mathsf{f}_{\omega}}(x(\omega))$.
Then \ref{t:2iii} asserts that $q\in\GG$ and
$q=\prox_{\gamma f}x$.
Hence, we derive from \ref{t:2i},
\cite[Remark~12.24]{Livre1}, and Definition~\ref{d:5} that
\begin{align}
\moyo{f}{\gamma}(x)
&=f\brk!{\prox_{\gamma f}x}
+(2\gamma)^{-1}\norm{x-\prox_{\gamma f}x}_{\HH}^2
\nonumber\\
&=\int_{\Omega}\brk2{\mathsf{f}_{\omega}\brk!{q(\omega)}+
(2\gamma)^{-1}\norm{x(\omega)-q(\omega)}_{\HW}^2}\mu(d\omega)
\nonumber\\
&=\int_{\Omega}
\moyo{\mathsf{f}_{\omega}}{\gamma}\brk!{x(\omega)}\mu(d\omega),
\end{align}
as claimed.

\ref{t:2viii}:
Let $x^*\in\GG$, let $(\gamma_n)_{n\in\NN}$ be a sequence in
$\zeroun$ such that $\gamma_n\downarrow 0$, and define
\begin{equation}
(\forall n\in\NN)\quad\vartheta_n\colon\Omega\to\RR\colon
\omega\mapsto\moyo{\brk!{\mathsf{f}_{\omega}^*}}{\gamma_n}
\brk!{x^*(\omega)}.
\end{equation}
For every $n\in\NN$, since Moreau's decomposition theorem
\cite[Theorem~14.3(i)]{Livre1} gives
\begin{equation}
\label{e:6ccc}
(\forall\omega\in\Omega)\quad\vartheta_n(\omega)=
\gamma_n^{-1}\norm{x^*(\omega)}_{\HW}^2-
\moyo{\mathsf{f}_{\omega}}{\gamma_n^{-1}}\brk!{
\gamma_n^{-1}x^*(\omega)},
\end{equation}
it follows from \ref{t:2vii} that
$\vartheta_n\in\mathscr{L}^1(\Omega,\FF,\mu;\RR)$.
Further, we deduce from \cite[Proposition~12.33(ii)]{Livre1} that
\begin{equation}
\label{e:idrz}
(\forall\omega\in\Omega)\quad
\brk!{\vartheta_n(\omega)}_{n\in\NN}\,\,
\text{is increasing and}\,\,
\vartheta_n(\omega)\uparrow
\mathsf{f}_{\omega}^*\brk!{x^*(\omega)}
\end{equation}
and, therefore, that the function $\Omega\to\RX\colon
\omega\mapsto\mathsf{f}_{\omega}^*(x^*(\omega))$ is
$\FF$-measurable. On the other hand, invoking \eqref{e:6ccc},
\ref{t:2vii}, Moreau's decomposition theorem, and
\cite[Proposition~12.33(ii)]{Livre1}, we obtain
\begin{align}
\int_{\Omega}\vartheta_n(\omega)\mu(d\omega)
&=\int_{\Omega}\brk2{
\gamma_n^{-1}\norm{x^*(\omega)}_{\HW}^2-
\moyo{\mathsf{f}_{\omega}}{\gamma_n^{-1}}\brk!{
\gamma_n^{-1}x^*(\omega)}}\mu(d\omega)
\nonumber\\
&=\gamma_n^{-1}\norm{x^*}_{\HH}^2-
\moyo{f}{\gamma_n^{-1}}\brk!{\gamma_n^{-1}x^*}
\nonumber\\
&=\moyo{(f^*)}{\gamma_n}(x^*)
\nonumber\\
&\to f^*(x^*)\,\,\text{as}\,\,n\to\pinf.
\end{align}
Thus, in view of \eqref{e:idrz}, we infer from the Beppo Levi
monotone convergence theorem 
\cite[Theorem~2.8.2 and Corollary~2.8.6]{Boga07} that
\begin{equation}
f^*(x^*)
=\lim\int_{\Omega}\vartheta_n(\omega)\mu(d\omega)
=\int_{\Omega}\lim\vartheta_n(\omega)\,\mu(d\omega)
=\int_{\Omega}
\mathsf{f}_{\omega}^*\brk!{x^*(\omega)}\mu(d\omega).
\end{equation}

\ref{t:2ix}:
Assumption~\ref{a:3}\ref{a:3c} ensures that
$(\forall\omega\in\Omega)$ $r(\omega)\in\dom\mathsf{f}_{\omega}$.
Now take $x\in\HH$ and set
\begin{equation}
\brk!{\forall\alpha\in\RPP}\quad
\theta_{\alpha}\colon\Omega\to\RX\colon\omega\mapsto
\frac{\mathsf{f}_{\omega}\brk!{r(\omega)+\alpha
x(\omega)}-\mathsf{f}_{\omega}\brk!{r(\omega)}}{\alpha}.
\end{equation}
Then, for every $\alpha\in\RPP$, since $r+\alpha x\in\GG$ and
$r\in\GG$, it results from \ref{t:2i-} that $\theta_{\alpha}$ is
$\FF$-measurable. On the other hand, by
Assumption~\ref{a:3}\ref{a:3a} and
\cite[Propositions~9.27 and 9.30(ii)]{Livre1}, we obtain
\begin{equation}
(\forall\omega\in\Omega)\quad
\text{the net}\,\,\brk!{\theta_{\alpha}(\omega)}_{\alpha\in\RPP}
\,\,\text{is increasing and}\,\,
(\rec\mathsf{f}_{\omega})\brk!{x(\omega)}
=\lim_{\alpha\uparrow\pinf}\theta_{\alpha}(\omega).
\end{equation}
Altogether, we infer from the Beppo Levi monotone convergence
theorem, Assumption~\ref{a:3}\ref{a:3c}, \ref{t:2i},
and \cite[Proposition~9.30(ii)]{Livre1} that
\begin{align}
\int_{\Omega}(\rec\mathsf{f}_{\omega})\brk!{x(\omega)}
\mu(d\omega)
&=\int_{\Omega}\lim_{\alpha\uparrow\pinf}
\theta_{\alpha}(\omega)\,\mu(d\omega)
\nonumber\\
&=\lim_{\alpha\uparrow\pinf}\int_{\Omega}
\theta_{\alpha}(\omega)\mu(d\omega)
\nonumber\\
&=\lim_{\alpha\uparrow\pinf}\frac{1}{\alpha}\brk3{
\int_{\Omega}\mathsf{f}_{\omega}\brk!{r(\omega)+\alpha
x(\omega)}\mu(d\omega)-
\int_{\Omega}\mathsf{f}_{\omega}\brk!{r(\omega)}\mu(d\omega)}
\nonumber\\
&=\lim_{\alpha\uparrow\pinf}\frac{f(r+\alpha x)-f(r)}{\alpha}
\nonumber\\
&=(\rec f)(x),
\end{align}
which completes the proof.
\end{proof}

\begin{remark}
\label{r:3}
Consider Theorem~\ref{t:2} in the special case of
Example~\ref{ex:1}\ref{ex:1iii}. Then
\ref{t:2i},
\ref{t:2ii},
\ref{t:2iii}, 
\ref{t:2viii}, and
\ref{t:2ix} were obtained, respectively, in
\cite[Corollary, p.~227]{Rock71},
\cite[Equation~(25)]{Rock71},
\cite[Proposition~24.13]{Livre1},
\cite[Theorem~2]{Rock71},
and \cite[Proposition~II.1]{Bism73}.
On the other hand, in the special case of
Example~\ref{ex:1}\ref{ex:1ii},
\ref{t:2iii} was obtained in
\cite[Corollary~2.2]{Ibap21}.
\end{remark}

\begin{example}
\label{ex:8}
Consider the setting of Example~\ref{ex:1}\ref{ex:1ii}
and suppose, in addition, that $(\forall k\in\NN)$
$\alpha_k=1$ and $\mathsf{H}_k=\RR$. Then $\HH=\ell^2(\NN)$.
Now set $(\forall k\in\NN)$ $\mathsf{f}_k=\abs{\Cdot}$. Then
\begin{equation}
\dom\brk3{\leftindex^{\mathfrak{G}}{\int}_{\Omega}^{\oplus}
\mathsf{f}_{\omega}\mu(d\omega)}
=\ell^1(\NN)
\neq
\ell^2(\NN)
=\leftindex^{\mathfrak{G}}{\int}_{\Omega}^{\oplus}
\dom\mathsf{f}_{\omega}\,\mu(d\omega).
\end{equation}
Thus, the closure operation in Theorem~\ref{t:2}\ref{t:2iv}
must not be omitted.
\end{example}

Every maximally monotone operator on $\RR$ is the subdifferential
of a function in $\Gamma_0(\RR)$ \cite[Corollary~22.23]{Livre1}.
The following result is an extension of this fact.

\begin{corollary}
\label{c:5}
Let $(\Omega,\FF,\mu)$ be a complete $\sigma$-finite measure space
and, for every $\omega\in\Omega$, let $\AW\colon\RR\to 2^{\RR}$ be
maximally monotone. Set $\HH=L^2(\Omega,\FF,\mu;\RR)$ and
\begin{equation}
\label{e:3hcz}
A\colon\HH\to 2^{\HH}\colon x\mapsto
\menge{x^*\in\HH}{(\forallmu\omega\in\Omega)\,\,
x^*(\omega)\in\AW\brk!{x(\omega)}}.
\end{equation}
Then the following are equivalent:
\begin{enumerate}
\item
\label{c:5i}
$A$ is maximally monotone.
\item
\label{c:5ii}
There exists $f\in\Gamma_0(\HH)$ such that $A=\partial f$.
\item
\label{c:5iii}
$\dom A\neq\emp$ and, for every $\mathsf{x}\in\RR$, the function
$\Omega\to\RR\colon\omega\mapsto J_{\AW}\mathsf{x}$ is
$\FF$-measurable.
\end{enumerate}
\end{corollary}
\begin{proof}
\ref{c:5ii}$\Rightarrow$\ref{c:5i}:
Use Moreau's theorem \cite[Proposition~12.b]{More65}.

\ref{c:5i}$\Rightarrow$\ref{c:5iii}:
This is a special case of Corollary~\ref{c:6}.

\ref{c:5iii}$\Rightarrow$\ref{c:5ii}:
Set $\mathfrak{G}=\menge{x\colon\Omega\to\RR}{\text{$x$ is
$\FF$-measurable}}$. Then, as seen in
Example~\ref{ex:1}\ref{ex:1iii},
$\HH=\leftindex^{\mathfrak{G}}{\int}_{\Omega}^{\oplus}
\RR\,\mu(d\omega)$. For every $\omega\in\Omega$,
\cite[Corollary~22.23]{Livre1} asserts that there exists
$\mathsf{g}_{\omega}\in\Gamma_0(\RR)$ such that
$\AW=\partial\mathsf{g}_{\omega}$.
Next, since $\dom A\neq\emp$ and $(\Omega,\FF,\mu)$ is complete,
there exist $r$ and $s^*$ in $\mathscr{L}^2(\Omega,\FF,\mu;\RR)$
such that
\begin{equation}
\label{e:keji}
(\forallmu\omega\in\Omega)\quad
s^*(\omega)
\in\AW\brk!{r(\omega)}
=\partial\mathsf{g}_{\omega}\brk!{r(\omega)}
\end{equation}
and
\begin{equation}
\label{e:kejj}
(\forall\omega\in\Omega)\quad
r(\omega)\in\dom\AW\subset\dom\mathsf{g}_{\omega}.
\end{equation}
Now set
\begin{equation}
\label{e:v0r7}
(\forall\omega\in\Omega)\quad
\mathsf{f}_{\omega}
=\mathsf{g}_{\omega}
-\mathsf{g}_{\omega}\brk!{r(\omega)}.
\end{equation}
Then the functions $(\mathsf{f}_{\omega})_{\omega\in\Omega}$ lie
in $\Gamma_0(\RR)$ and, by
\cite[Proposition~24.8(i) and Example~23.3]{Livre1},
$(\forall\omega\in\Omega)$
$\prox_{\mathsf{f}_{\omega}}
=\prox_{\mathsf{g}_{\omega}}
=J_{\AW}$.
In turn, appealing to the continuity of the operators
$(J_{\AW})_{\omega\in\Omega}$, we deduce from
\cite[Lemma~III.14]{Cast77}
that the mapping $\Omega\times\RR\to\RR\colon(\omega,\mathsf{x})
\mapsto\prox_{\mathsf{f}_{\omega}}\mathsf{x}$
is $\FF\otimes\BE_{\RR}$-measurable.
Therefore, for every $x\in\mathscr{L}^2(\Omega,\FF,\mu;\RR)$,
the mapping $\Omega\to\RR\colon
\omega\mapsto\prox_{\mathsf{f}_{\omega}}(x(\omega))$ lies
in $\mathfrak{G}$. Next, we get from \eqref{e:v0r7}
and \eqref{e:kejj} that $(\forall\omega\in\Omega)$
$\mathsf{f}_{\omega}(r(\omega))=0$. Moreover, by \eqref{e:keji},
\begin{equation}
(\forallmu\omega\in\Omega)(\forall\mathsf{x}\in\RR)\quad
\mathsf{f}_{\omega}(\mathsf{x})
=\mathsf{g}_{\omega}(\mathsf{x})-
\mathsf{g}_{\omega}\brk!{r(\omega)}
\geq \brk!{\mathsf{x}-r(\omega)}s^*(\omega)
=\mathsf{x}s^*(\omega)-r(\omega)s^*(\omega).
\end{equation}
Hence, since $\omega\mapsto r(\omega)s^*(\omega)$ lies in
$\mathscr{L}^1(\Omega,\FF,\mu;\RR)$, the family
$(\mathsf{f}_{\omega})_{\omega\in\Omega}$ satisfies the assumption
of Theorem~\ref{t:2}. Altogether, we conclude via
Theorem~\ref{t:2}\ref{t:2i} that
\begin{equation}
\leftindex^{\mathfrak{G}}{\int}_{\Omega}^{\oplus}
\mathsf{f}_{\omega}\mu(d\omega)\in\Gamma_0(\HH)
\end{equation}
and via Theorem~\ref{t:2}\ref{t:2ii} and \eqref{e:3hcz} that
\begin{equation}
\partial\brk3{\leftindex^{\mathfrak{G}}{\int}_{\Omega}^{\oplus}
\mathsf{f}_{\omega}\mu(d\omega)}
=\leftindex^{\mathfrak{G}}{\Int}_{\Omega}^{\oplus}
\partial\mathsf{f}_{\omega}\mu(d\omega)
=\leftindex^{\mathfrak{G}}{\Int}_{\Omega}^{\oplus}
\AW\mu(d\omega)
=A,
\end{equation}
as desired.
\end{proof}

\begin{corollary}
Let $(\mathsf{A}_k)_{k\in\NN}$ be a family of maximally monotone
operators from $\RR$ to $2^{\RR}$, and define
\begin{equation}
A\colon\ell^2(\NN)\to 2^{\ell^2(\NN)}\colon
(\mathsf{x}_k)_{k\in\NN}\mapsto
\menge{(\mathsf{x}_k^*)_{k\in\NN}\in\ell^2(\NN)}{
(\forall k\in\NN)\,\,\mathsf{x}_k^*\in
\mathsf{A}_k\mathsf{x}_k}.
\end{equation}
Suppose that $\dom A\neq\emp$. Then $A$ is maximally monotone
and there exists $f\in\Gamma_0(\ell^2(\NN))$ such that
$A=\partial f$.
\end{corollary}
\begin{proof}
Apply Corollary~\ref{c:5} to the case where $\Omega=\NN$,
$\FF=2^{\NN}$, and $\mu$ is the counting measure.
\end{proof}

\begin{corollary}
\label{c:1}
Suppose that Assumption~\ref{a:1} is in force and, for every
$\omega\in\Omega$, let $\mathsf{C}_{\omega}$ be a nonempty closed
convex subset of $\HW$. Set
\begin{equation}
\label{e:dbih}
C=\leftindex^{\mathfrak{G}}{\int}_{\Omega}^{\oplus}
\mathsf{C}_{\omega}\mu(d\omega).
\end{equation}
Suppose that $C\neq\emp$ and that, for every $x\in\GG$,
the mapping $\omega\mapsto\proj_{\mathsf{C}_{\omega}}(x(\omega))$
lies in $\mathfrak{G}$. Then the following hold:
\begin{enumerate}
\item
\label{c:1i}
$C$ is a closed convex subset of $\HH$.
\item
\label{c:1ii}
$N_C=\leftindex^{\mathfrak{G}}{\Int}_{\Omega}^{\oplus}
N_{\mathsf{C}_{\omega}}\mu(d\omega)$.
\item
\label{c:1iii}
$\proj_C=\leftindex^{\mathfrak{G}}{\Int}_{\Omega}^{\oplus}
\proj_{\mathsf{C}_{\omega}}\mu(d\omega)$.
\item
\label{c:1iv}
$d_C^2=\leftindex^{\mathfrak{G}}{\Int}_{\Omega}^{\oplus}
d_{\mathsf{C}_{\omega}}^2\mu(d\omega)$.
\item
\label{c:1v}
$\sigma_C=\leftindex^{\mathfrak{G}}{\Int}_{\Omega}^{\oplus}
\sigma_{\mathsf{C}_{\omega}}\mu(d\omega)$.
\item
\label{c:1vi}
Suppose that, for every $\omega\in\Omega$,
$\mathsf{C}_{\omega}$ is a cone in $\HW$. Then
$C^{\ominus}=\leftindex^{\mathfrak{G}}{\Int}_{\Omega}^{\oplus}
\mathsf{C}_{\omega}^{\ominus}\,\mu(d\omega)$.
\item
\label{c:1vii}
Suppose that, for every $\omega\in\Omega$,
$\mathsf{C}_{\omega}$ is a vector subspace of $\HW$. Then
$C^{\perp}=\leftindex^{\mathfrak{G}}{\Int}_{\Omega}^{\oplus}
\mathsf{C}_{\omega}^{\perp}\,\mu(d\omega)$.
\end{enumerate}
\end{corollary}
\begin{proof}
Set $(\forall\omega\in\Omega)$
$\mathsf{f}_{\omega}=\iota_{\mathsf{C}_{\omega}}$.
Then, for every $\omega\in\Omega$,
$\mathsf{f}_{\omega}\in\Gamma_0(\HW)$,
$\mathsf{f}_{\omega}\geq 0$, and
$\prox_{\mathsf{f}_{\omega}}=\proj_{\mathsf{C}_{\omega}}$.
Moreover, since $C\neq\emp$ and $(\Omega,\FF,\mu)$ is complete,
there exists $r\in\GG$ such that, for every $\omega\in\Omega$,
$r(\omega)\in\mathsf{C}_{\omega}$ or, equivalently,
$\mathsf{f}_{\omega}(r(\omega))=0$. Altogether, the
family $(\mathsf{f}_{\omega})_{\omega\in\Omega}$ satisfies the
assumption of Theorem~\ref{t:2}. Therefore, in view of items
\ref{t:2i-} and \ref{t:2i} in Theorem~\ref{t:2},
\begin{equation}
\label{e:8mnt}
f=\leftindex^{\mathfrak{G}}{\int}_{\Omega}^{\oplus}
\mathsf{f}_{\omega}\mu(d\omega)\,\,
\text{is well defined and lies in}\,\,\Gamma_0\brk!{\HH}.
\end{equation}

\ref{c:1i}:
Using Definitions~\ref{d:5} and \ref{d:3}, together with
\eqref{e:dbih}, we obtain
\begin{align}
\label{e:jwk7}
(\forall x\in\HH)\quad
f(x)
&=\int_{\Omega}\iota_{\mathsf{C}_{\omega}}\brk!{x(\omega)}
\mu(d\omega)
\nonumber\\
&=
\begin{cases}
0,&\text{if}\,\,(\forallmu\omega\in\Omega)\,\,x(\omega)\in
\mathsf{C}_{\omega};\\
\pinf,&\text{otherwise}
\end{cases}
\nonumber\\
&=
\begin{cases}
0,&\text{if}\,\,x\in C;\\
\pinf,&\text{otherwise}
\end{cases}
\nonumber\\
&=\iota_C(x),
\end{align}
and the claim thus follows from \eqref{e:8mnt}.

\ref{c:1ii}--\ref{c:1v}:
In the light of \eqref{e:8mnt} and \eqref{e:jwk7},
these follow from items \ref{t:2ii}, \ref{t:2iii},
\ref{t:2vii}, and \ref{t:2viii} in Theorem~\ref{t:2}, respectively.

\ref{c:1vi}:
We deduce from \cite[Example~6.40]{Livre1} and \ref{c:1ii} that
\begin{equation}
C^{\ominus}
=N_C0
=\leftindex^{\mathfrak{G}}{\Int}_{\Omega}^{\oplus}
\brk!{N_{\mathsf{C}_{\omega}}\mathsf{0}}\mu(d\omega)
=\leftindex^{\mathfrak{G}}{\Int}_{\Omega}^{\oplus}
\mathsf{C}_{\omega}^{\ominus}\,\mu(d\omega).
\end{equation}

\ref{c:1vii}:
Use \ref{c:1vi} and \cite[Proposition~6.23]{Livre1}.
\end{proof}

\begin{proposition}
\label{p:90}
Suppose that Assumption~\ref{a:3} is in force.
Let $\mathsf{G}$ be a separable real Hilbert space and,
for every $\omega\in\Omega$, let
$\LW\colon\mathsf{G}\to\HW$ be linear and bounded.
Suppose that, for every $\mathsf{z}\in\mathsf{G}$, the mapping
$\mathfrak{e}_{\mathsf{L}}\mathsf{z}\colon
\omega\mapsto\LW\mathsf{z}$
lies in $\mathfrak{G}$. Additionally, suppose that
$\int_{\Omega}\norm{\LW}^2\mu(d\omega)<\pinf$
and that there exists $\mathsf{w}\in\mathsf{G}$ such that
$\int_{\Omega}\mathsf{f}_{\omega}
\brk{\LW\mathsf{w}}\mu(d\omega)<\pinf$.
Define
\begin{equation}
\label{e:chrp}
\mathsf{g}\colon\mathsf{G}\to\RX\colon\mathsf{z}\mapsto
\int_{\Omega}\mathsf{f}_{\omega}
\brk!{\LW\mathsf{z}}\mu(d\omega).
\end{equation}
Then the following hold:
\begin{enumerate}
\item
\label{p:90i}
$\mathsf{g}$ is well defined and lies in $\Gamma_0(\mathsf{G})$.
\item
\label{p:90ii}
Let $(\mathsf{z},\mathsf{z}^*)\in\mathsf{G}\times\mathsf{G}$.
Then $\mathsf{z}^*\in\partial\mathsf{g}(\mathsf{z})$
if and only if there exist sequences
$(\gamma_n)_{n\in\NN}$ in $\RPP$ and
$(\mathsf{z}_n)_{n\in\NN}$ in
$\mathsf{G}$ such that
\begin{equation}
\gamma_n\downarrow 0,\quad
\mathsf{z}_n\to\mathsf{z},\quad\text{and}\quad
\int_{\Omega}\LW^*\brk2{
\prox_{\gamma_n^{-1}\mathsf{f}_{\omega}^*}\brk!{\gamma_n^{-1}
\LW\mathsf{z}_n}}\mu(d\omega)
\to\mathsf{z}^*.
\end{equation}
\end{enumerate}
\end{proposition}
\begin{proof}
Theorem~\ref{t:2}\ref{t:2i-}--\ref{t:2i} state that
\begin{equation}
\label{e:frm5}
f=\leftindex^{\mathfrak{G}}{\int}_{\Omega}^{\oplus}
\mathsf{f}_{\omega}\mu(d\omega)\,\,
\text{is well defined and lies in}\,\,\Gamma_0\brk!{\HH}.
\end{equation}
On the other hand, according to
Proposition~\ref{p:9}\ref{p:9i},
\begin{equation}
\label{e:zc2u}
L\colon\mathsf{G}\to\HH\colon\mathsf{z}\mapsto
\mathfrak{e}_{\mathsf{L}}\mathsf{z}
\,\,\text{is well defined, linear, and bounded}.
\end{equation}

\ref{p:90i}:
Because $L\mathsf{w}\in\dom f$, it follows from \eqref{e:chrp},
\eqref{e:frm5}, and \eqref{e:zc2u} that
\begin{equation}
\label{e:99o2}
\mathsf{g}=f\circ L\in\Gamma_0(\mathsf{G}).
\end{equation}

\ref{p:90ii}:
It results from Theorem~\ref{t:2}\ref{t:2ii},
Proposition~\ref{p:16}, and Moreau's decomposition
\cite[Theorem~14.3(ii)]{Livre1} that
\begin{equation}
\brk!{\forall\gamma\in\RPP}\quad
\moyo{(\partial f)}{\gamma}
=\leftindex^{\mathfrak{G}}{\Int}_{\Omega}^{\oplus}
\moyo{(\partial\mathsf{f}_{\omega})}{\gamma}\mu(d\omega)
=\leftindex^{\mathfrak{G}}{\Int}_{\Omega}^{\oplus}
\prox_{\gamma^{-1}\mathsf{f}_{\omega}}\circ\brk1{
\mathop{\gamma^{-1}\Id_{\HW}}}\mu(d\omega).
\end{equation}
Hence, for every $\gamma\in\RPP$ and every
$\mathsf{w}\in\mathsf{G}$, since
$\mathfrak{e}_{\mathsf{L}}\mathsf{w}\colon\omega\mapsto
\LW\mathsf{w}$ is a representative in $\GG$ of $L\mathsf{w}$,
Proposition~\ref{p:9}\ref{p:9iv} implies that
\begin{equation}
L^*\brk!{\moyo{(\partial f)}{\gamma}(L\mathsf{w})}
=\int_{\Omega}\LW^*\brk2{
\prox_{\gamma^{-1}\mathsf{f}_{\omega}^*}\brk!{\gamma^{-1}
\LW\mathsf{w}}}\mu(d\omega).
\end{equation}
In addition, appealing to \eqref{e:frm5}, \eqref{e:zc2u},
and \eqref{e:99o2}, we derive from \cite[Theorem~4.1]{Penn03} and
a remark on \cite[p.~88]{Penn03} that
$\gra\partial\mathsf{g}$ is the set of points
$(\mathsf{w},\mathsf{w}^*)\in\mathsf{G}\times\mathsf{G}$
for which there exist sequences
$(\gamma_n)_{n\in\NN}$ in $\RPP$ and
$(\mathsf{w}_n)_{n\in\NN}$ in $\mathsf{G}$ such that
$\gamma_n\downarrow 0$,
$\mathsf{w}_n\to\mathsf{w}$, and
$L^*(\moyo{(\partial f)}{\gamma_n}(L\mathsf{w}_n))
\to\mathsf{w}_n^*$. Altogether, the proof is complete.
\end{proof}

\section{Application to integral composite inclusion problems}
\label{sec:5}

Let $\mathsf{G}$ and $(\mathsf{H}_k)_{1\leq k\leq p}$ be real
Hilbert spaces. For every $k\in\{1,\ldots,p\}$, let
$\mathsf{A}_k\colon\mathsf{H}_k\to2^{\mathsf{H}_k}$ be monotone and
let $\mathsf{L}_k\colon\mathsf{G}\to\mathsf{H}_k$ be linear and
bounded. Finite compositions of the form
$\sum_{k=1}^p\mathsf{L}_k^*\circ\mathsf{A}_k\circ\mathsf{L}_k$
arise in many theoretical and modeling aspects of monotone operator
theory \cite{Livre1,Brow69,Acnu24,Ekel99,Ghou09}. The main object
of this section is to extend this construction to arbitrary
families of monotone and linear operators. More precisely, our
focus is on the following monotonicity-preserving operation, which
involves the Aumann integral of \eqref{e:aum3}.

\begin{proposition}
\label{p:5}
Suppose that Assumption~\ref{a:1} is in force.
Let $\mathsf{G}$ be a separable real Hilbert space and,
for every $\omega\in\Omega$, let
$\AW\colon\HW\to 2^{\HW}$ be monotone and let
$\LW\colon\mathsf{G}\to\HW$ be linear and bounded. Then
\begin{equation}
\label{e:kt2z}
\mathsf{M}\colon\mathsf{G}\to 2^{\mathsf{G}}\colon
\mathsf{z}\mapsto\int_{\Omega}\LW^*
\brk!{\AW(\LW\mathsf{z})}\mu(d\omega)
\end{equation}
is monotone.
\end{proposition}
\begin{proof}
Suppose that $(\mathsf{z},\mathsf{z}^*)$ and
$(\mathsf{w},\mathsf{w}^*)$ are in $\gra\mathsf{M}$. Then,
by \eqref{e:aum3}, there exist $x^*$ and $y^*$ in
$\prod_{\omega\in\Omega}\HW$ such that
\begin{equation}
\begin{cases}
(\forallmu\omega\in\Omega)\;\;
x^*(\omega)\in\AW(\LW\mathsf{z})\,\,\text{and}\,\,
y^*(\omega)\in\AW(\LW\mathsf{w})\\
\text{the mappings}\,\,
\omega\mapsto\LW^*\brk{x^*(\omega)}
\,\,\text{and}\,\,
\omega\mapsto\LW^*\brk{y^*(\omega)}
\,\,\text{lie in}\,\,
\mathscr{L}^1(\Omega,\FF,\mu;\mathsf{G})\\
\int_{\Omega}\LW^*\brk{x^*(\omega)}\mu(d\omega)=\mathsf{z}^*
\,\,\text{and}\,\,
\int_{\Omega}\LW^*\brk{y^*(\omega)}\mu(d\omega)=\mathsf{w}^*.
\end{cases}
\end{equation}
The monotonicity of the operators
$(\AW)_{\omega\in\Omega}$ ensures that
\begin{equation}
(\forallmu\omega\in\Omega)\quad
\scal!{\mathsf{z}-\mathsf{w}}{
\LW^*\brk!{x^*(\omega)}-\LW^*\brk!{y^*(\omega)}}_{\mathsf{G}}
=\scal{\LW\mathsf{z}-\LW\mathsf{w}}{x^*(\omega)-y^*(\omega)}_{\HW}
\geq 0.
\end{equation}
Therefore, using \cite[Th\'eor\`eme~5.8.16]{Sch93b}, we obtain
\begin{align}
\scal{\mathsf{z}-\mathsf{w}}{\mathsf{z}^*-
\mathsf{w}^*}_{\mathsf{G}}
&=\scal3{\mathsf{z}-\mathsf{w}}{
\int_{\Omega}\LW^*\brk!{x^*(\omega)}\mu(d\omega)-
\int_{\Omega}\LW^*\brk!{y^*(\omega)}\mu(d\omega)}_{\mathsf{G}}
\nonumber\\
&=\int_{\Omega}\scal!{\mathsf{z}-\mathsf{w}}{
\LW^*\brk!{x^*(\omega)}-\LW^*\brk!{y^*(\omega)}}_{\mathsf{G}}\,
\mu(d\omega)
\nonumber\\
&\geq 0,
\end{align}
which yields the assertion.
\end{proof}

The inclusion problem under investigation involves the integral
composite operator \eqref{e:kt2z} and is placed in the following
environment.

\begin{assumption}
\label{a:4}
Assumption~\ref{a:1} and the following are in force:
\begin{enumerate}[label={\normalfont[\Alph*]}]
\item
\label{a:4a}
$\mathsf{G}$ is a separable real Hilbert space.
\item
\label{a:4b}
For every $\omega\in\Omega$,
$\LW\colon\mathsf{G}\to\HW$ is linear and bounded.
\item
\label{a:4c}
For every $\mathsf{z}\in\mathsf{G}$, the mapping
$\mathfrak{e}_{\mathsf{L}}\mathsf{z}\colon
\omega\mapsto\LW\mathsf{z}$ lies in $\mathfrak{G}$.
\item
\label{a:4d}
$\int_{\Omega}\norm{\LW}^2\mu(d\omega)<\pinf$.
\end{enumerate}
\end{assumption}

\begin{problem}
\label{prob:1}
Suppose that Assumptions~\ref{a:2} and \ref{a:4} are in force,
and let $\mathsf{W}\colon\mathsf{G}\to 2^{\mathsf{G}}$ be
maximally monotone. The objective is to
\begin{equation}
\label{e:89p}
\text{find}\,\,\mathsf{z}\in\mathsf{G}\,\,\text{such that}\,\,
\mathsf{0}\in\mathsf{W}\mathsf{z}+
\int_{\Omega}\LW^*\brk!{\AW(\LW\mathsf{z})}\mu(d\omega).
\end{equation}
\end{problem}

In traditional variational methods, duality provides a powerful
framework to analyze and solve minimization problems 
\cite{Livre1,Ekel99,Rock74}.
More generally, for inclusion problems, 
notions of duality have been proposed at various levels of
generality \cite{Moor22,Svva12,Penn00,Robi99}
in the context of Example~\ref{ex:1}\ref{ex:1i+}, 
which corresponds to the inclusion problem
\begin{equation}
\label{e:1}
\text{find}\,\,\mathsf{z}\in\mathsf{G}\,\,\text{such that}\,\,
\mathsf{0}\in\mathsf{Wz}+\sum_{k=1}^p
\mathsf{L}_k^*\brk!{\mathsf{A}_k(\mathsf{L}_k\mathsf{z})}.
\end{equation}
The next theorem extends duality concepts to the general setting
of Problem~\ref{prob:1}.

\begin{theorem}
\label{t:35}
Consider the setting of Problem~\ref{prob:1},
as well as the \emph{dual problem}
\begin{equation}
\label{e:89d}
\text{find}\,\,x^*\in\HH\,\,\text{such that}\,\,
\brk3{\exi\mathsf{z}\in\mathsf{W}^{-1}\brk3{
{-}\int_{\Omega}\LW^*\brk!{x^*(\omega)}
\mu(d\omega)}}(\forallmu\omega\in\Omega)\,\,
\LW\mathsf{z}\in\AW^{-1}\brk!{x^*(\omega)},
\end{equation}
and denote by $\mathsf{Z}$ and $Z^*$ the sets of
solutions to \eqref{e:89p} and \eqref{e:89d}, respectively.
Let $\kut$ be the \emph{Kuhn--Tucker operator} associated
to Problem~\ref{prob:1}, that is,
\begin{equation}
\label{e:89kt}
\begin{array}{ccll}
\kut\colon
&\mathsf{G}\oplus\HH&\to &2^{\mathsf{G}\oplus\HH}\\
&\brk{\mathsf{z},x^*}&\mapsto&
\brk3{\mathsf{W}\mathsf{z}+\Int_{\Omega}
\LW^*\brk!{x^*(\omega)}\mu(d\omega)}\times
\brk3{-\mathfrak{e}_{\mathsf{L}}\mathsf{z}+
\leftindex^{\mathfrak{G}}{\Int}_{\Omega}^{\oplus}
\AW^{-1}\brk!{x^*(\omega)}\mu(d\omega)},
\end{array}
\end{equation}
and let $\sad$ be the \emph{saddle operator} associated to
Problem~\ref{prob:1}, that is,
\begin{equation}
\label{e:89sad}
\begin{array}{ccll}
\sad\colon
&\mathsf{G}\oplus\HH\oplus\HH&\to
&2^{\mathsf{G}\oplus\HH\oplus\HH}\\
&\brk{\mathsf{z},x,u^*}&\mapsto&
\brk3{\mathsf{W}\mathsf{z}+\Int_{\Omega}
\LW^*\brk!{u^*(\omega)}\mu(d\omega)}
\times\brk3{\leftindex^{\mathfrak{G}}{\Int}_{\Omega}^{\oplus}
\AW\brk!{x(\omega)}\mu(d\omega)-u^*}
\times\brk!{{-}\mathfrak{e}_{\mathsf{L}}\mathsf{z}+x}.
\end{array}
\end{equation}
Then the following hold:
\begin{enumerate}
\item
\label{t:35i}
$\kut$ and $\sad$ are maximally monotone.
\item
\label{t:35ii}
$\zer\kut$ and $\zer\sad$ are closed and convex.
\item
\label{t:35iii}
Let $(\mathsf{z},x^*)\in\mathsf{G}\times\HH$. Then
$(\mathsf{z},x^*)\in\zer\kut$
$\Rightarrow$
$(\mathsf{z},x^*)\in\mathsf{Z}\times Z^*$.
\item
\label{t:35iv}
Let $(\mathsf{z},x,u^*)\in\mathsf{G}\times\HH\times\HH$. Then
$(\mathsf{z},x,u^*)\in\zer\sad$
$\Rightarrow$
$(\mathsf{z},u^*)\in\mathsf{Z}\times Z^*$.
\item
\label{t:35v}
$\zer\sad\neq\emp$
$\Leftrightarrow$
$\zer\kut\neq\emp$
$\Leftrightarrow$
$Z^*\neq\emp$
$\Rightarrow$
$\mathsf{Z}\neq\emp$.
\end{enumerate}
\end{theorem}
\begin{proof}
Set
\begin{equation}
\label{e:84jb}
A=\leftindex^{\mathfrak{G}}{\int}_{\Omega}^{\oplus}
\AW\mu(d\omega).
\end{equation}
Theorem~\ref{t:1}\ref{t:1i} states that
\begin{equation}
\label{e:xgy0}
\text{$A$ is maximally monotone},
\end{equation}
while Proposition~\ref{p:1}\ref{p:1iv} states that
\begin{equation}
A^{-1}=\leftindex^{\mathfrak{G}}{\int}_{\Omega}^{\oplus}
\AW^{-1}\mu(d\omega).
\end{equation}
Moreover, in view of Assumption~\ref{a:4},
items \ref{p:9i} and \ref{p:9iv} of
Proposition~\ref{p:9} imply that the operator
\begin{equation}
\label{e:9hw1}
L\colon\mathsf{G}\to\HH\colon\mathsf{z}\mapsto
\mathfrak{e}_{\mathsf{L}}\mathsf{z}
\end{equation}
is well defined, linear, and bounded, with adjoint 
\begin{equation}
\label{e:aavw}
L^*\colon\HH\to\mathsf{G}\colon x^*\mapsto
\int_{\Omega}\LW^*\brk!{x^*(\omega)}\mu(d\omega).
\end{equation}
Hence, we deduce from \eqref{e:89kt} that
\begin{equation}
\label{e:o4kut}
\kut\colon\mathsf{G}\oplus\HH\to 2^{\mathsf{G}\oplus\HH}\colon
(\mathsf{z},x^*)\mapsto
\brk!{\mathsf{W}\mathsf{z}+L^*x^*}\times
\brk!{{-}L\mathsf{z}+A^{-1}x^*}
\end{equation}
and from \eqref{e:89sad} that
\begin{equation}
\label{e:o4sad}
\sad\colon\mathsf{G}\oplus\HH\oplus\HH\to
2^{\mathsf{G}\oplus\HH\oplus\HH}\colon
(\mathsf{z},x,u^*)\mapsto
\brk!{\mathsf{W}\mathsf{z}+L^*u^*}\times
(Ax-u^*)\times\brk[c]!{{-}L\mathsf{z}+x}.
\end{equation}
Additionally, the dual problem \eqref{e:89d} can be rewritten as
\begin{equation}
\label{e:o4dd}
\text{find}\,\,x^*\in\HH\,\,\text{such that}\,\,
0\in{-}L\brk!{\mathsf{W}^{-1}(-L^*x^*)}+A^{-1}x^*.
\end{equation}

\ref{t:35i}:
In view of \eqref{e:xgy0} and the maximal monotonicity of
$\mathsf{W}$, it follows from
\eqref{e:o4kut} and
\cite[Proposition~26.32(iii)]{Livre1}
that $\kut$ is maximally monotone, and from
\eqref{e:o4sad} and \cite[Lemma~2.2(ii)]{Jmaa20} that
$\sad$ is maximally monotone.

\ref{t:35ii}:
Combine \ref{t:35i} and \cite[Proposition~23.39]{Livre1}.

\ref{t:35iii}:
Suppose that $(\mathsf{z},x^*)\in\zer\kut$. Then, by
\eqref{e:o4kut}, $L\mathsf{z}\in A^{-1}x^*$ or, equivalently,
$x^*\in A(L\mathsf{z})$. Therefore, it follows from
\eqref{e:9hw1}, Assumption~\ref{a:4}\ref{a:4c}, and
\eqref{e:84jb} that, for $\mu$-almost every $\omega\in\Omega$,
$x^*(\omega)\in\AW(\LW\mathsf{z})$ and, in turn, that
$\LW^*(x^*(\omega))\in\LW^*(\AW(\LW\mathsf{z}))$.
Hence, because Proposition~\ref{p:9}\ref{p:9iii}
asserts that the mapping
$\Omega\to\mathsf{G}\colon\omega\mapsto\LW^*(x^*(\omega))$
is $\mu$-integrable, we infer from
\eqref{e:89kt} and \eqref{e:aum3} that
\begin{equation}
\mathsf{0}
\in\mathsf{W}\mathsf{z}+
\int_{\Omega}\LW^*\brk!{x^*(\omega)}\mu(d\omega)
\subset\mathsf{W}\mathsf{z}+
\int_{\Omega}\LW^*\brk!{\AW(\LW\mathsf{z})}\mu(d\omega).
\end{equation}
Finally, since $(\mathsf{z},x^*)\in\zer\kut$,
it follows from \cite[Proposition~26.33(ii)]{Livre1} that $x^*$
solves \eqref{e:o4dd} and, therefore, \eqref{e:89d}.

\ref{t:35iv}:
Argue as in \ref{t:35iii}.

\ref{t:35v}:
By virtue of \eqref{e:o4kut}, \eqref{e:o4sad}, and \eqref{e:o4dd},
the equivalences $\zer\sad\neq\emp$ $\Leftrightarrow$
$\zer\kut\neq\emp$ $\Leftrightarrow$ $Z^*\neq\emp$
follow from \cite[Lemma~2.2(iv)]{Jmaa20}, while the implication
$\zer\kut\neq\emp$ $\Rightarrow$
$\mathsf{Z}\neq\emp$ follows from
\ref{t:35iii}.
\end{proof}

\begin{remark}
\label{r:2}
Consider the setting of Theorem~\ref{t:35}, and define
$A$ as in \eqref{e:84jb} and $L$ as in \eqref{e:9hw1}.
\begin{enumerate}
\item
$\zer(\mathsf{W}+L^*\circ A\circ L)$ is a subset of
$\mathsf{Z}$ which, in general, is proper.
\item
According to Theorem~\ref{t:35}\ref{t:35iii}--\ref{t:35iv},
to solve \eqref{e:89p} and its dual \eqref{e:89d},
it is enough to find a zero of the operator $\kut$ of
\eqref{e:89kt} or of the operator $\sad$ of \eqref{e:89sad}. 
This can be achieved by using splitting algorithms \cite{Acnu24}. 
For instance, to find a zero of $\sad$, each operator
$\AW$ is decomposed as $\AW=\AW^{\MM}+\AW^{\CC}+\AW^{\LL}$,
where $\AW^{\MM}\colon\HW\to 2^{\HW}$ is maximally monotone,
$\AW^{\CC}\colon\HW\to\HW$ is cocoercive, and
$\AW^{\LL}\colon\HW\to\HW$ is monotone and Lipschitzian.
Thus, $A$ is decomposed as
\begin{equation}
A=
\leftindex^{\mathfrak{G}}{\int}_{\Omega}^{\oplus}
\AW^{\MM}\mu(d\omega)+
\leftindex^{\mathfrak{G}}{\int}_{\Omega}^{\oplus}
\AW^{\CC}\mu(d\omega)+
\leftindex^{\mathfrak{G}}{\int}_{\Omega}^{\oplus}
\AW^{\LL}\mu(d\omega).
\end{equation}
One can then employ the algorithm of \cite[Section~8.5]{Acnu24}.
It requires the resolvent of
$\leftindex^{\mathfrak{G}}{\int}_{\Omega}^{\oplus}
\AW^{\MM}\mu(d\omega)$, which can be implemented via
Theorem~\ref{t:1}\ref{t:1iia}, as well as Euler steps on
$\leftindex^{\mathfrak{G}}{\int}_{\Omega}^{\oplus}
\AW^{\CC}\mu(d\omega)$ and
$\leftindex^{\mathfrak{G}}{\int}_{\Omega}^{\oplus}
\AW^{\LL}\mu(d\omega)$, which can be implemented via
Proposition~\ref{p:2}\ref{p:2i}--\ref{p:2ii}.
\end{enumerate}
\end{remark}

We conclude the paper by providing a few illustrations of
Problem~\ref{prob:1} and the proposed duality framework;
see \cite{Jota24} for further applications.

\begin{example}
\label{ex:40}
In the setting of Example~\ref{ex:1}\ref{ex:1i+}, the primal
inclusion \eqref{e:89p} reduces to \eqref{e:1} and
Theorem~\ref{t:35} specializes to results found in
\cite[Proposition~1]{Moor22}.
\end{example}

\begin{example}
Suppose that Assumptions~\ref{a:3} and \ref{a:4} are in force,
let $\mathsf{g}\in\Gamma_0(\mathsf{G})$, and 
suppose that there exists $z^*\in\GG$ such that
\begin{equation}
\brk3{\exi\mathsf{w}\in
\partial\mathsf{g}^*\brk3{{-}\int_{\Omega}
\LW^*\brk!{z^*(\omega)}
\mu(d\omega)}}(\forallmu\omega\in\Omega)\quad
\LW\mathsf{w}\in
\partial\mathsf{f}_{\omega}^*\brk!{z^*(\omega)}.
\end{equation}
Now set $\mathsf{W}=\partial\mathsf{g}$ and
$(\forall\omega\in\Omega)$ $\AW=\partial\mathsf{f}_{\omega}$.
Then it follows from Theorem~\ref{t:2}, Proposition~\ref{p:9},
and standard convex calculus that every solution to the primal
problem \eqref{e:89p} solves
\begin{equation}
\label{e:2eep}
\minimize{\mathsf{z}\in\mathsf{G}}{\mathsf{g}(\mathsf{z})+
\int_{\Omega}\mathsf{f}_{\omega}
\brk!{\LW\mathsf{z}}\mu(d\omega)},
\end{equation}
and every solution to the dual problem \eqref{e:89d} solves
\begin{equation}
\label{e:2eed}
\minimize{x^*\in\HH}{\mathsf{g}^*\brk3{{-}\int_{\Omega}
\LW^*\brk!{x^*(\omega)}\mu(d\omega)}+
\int_{\Omega}\mathsf{f}_{\omega}^*\brk!{x^*(\omega)}\mu(d\omega)}.
\end{equation}
A noteworthy instance is when $\mu$ is a probability measure
and, for every $\omega\in\Omega$, $\HW=\mathsf{G}$ and
$\LW=\Id_{\mathsf{G}}$. In this setting, \eqref{e:2eep} describes
a standard stochastic optimization problem \cite{Nemi09}. Our
setting makes it possible to extend such stochastic problems to
composite ones involving functions acting on different spaces
$(\HW)_{\omega\in\Omega}$.
\end{example}

\begin{example}
Suppose that Assumption~\ref{a:4} is in force,
let $\mathsf{W}\colon\mathsf{G}\to 2^{\mathsf{G}}$ be
maximally monotone, and,
for every $\omega\in\Omega$, let $\BW\colon\HW\to 2^{\HW}$
be maximally monotone. Additionally, suppose that
$\dom\leftindex^{\mathfrak{G}}{\int}_{\Omega}^{\oplus}
\BW\mu(d\omega)\neq\emp$ and that, for every $x\in\GG$,
the mapping $\omega\mapsto J_{\BW}(x(\omega))$ lies in
$\mathfrak{G}$. Now let $\gamma\in\RPP$ and set
$(\forall\omega\in\Omega)$ $\AW=\moyo{\BW}{\gamma}$.
Then, by Theorem~\ref{t:1}\ref{t:1iib} and
Proposition~\ref{p:2}\ref{p:2ii}, the family
$(\AW)_{\omega\in\Omega}$ satisfies Assumption~\ref{a:2}. Further,
the primal problem \eqref{e:89p} becomes
\begin{equation}
\label{e:w7q3}
\text{find}\,\,\mathsf{z}\in\mathsf{G}\,\,\text{such that}\,\,
\mathsf{0}\in\mathsf{W}\mathsf{z}+
\int_{\Omega}\LW^*\brk!{\moyo{\BW}{\gamma}
(\LW\mathsf{z})}\mu(d\omega),
\end{equation}
and the dual problem \eqref{e:89d} reads
\begin{multline}
\label{e:w7qd}
\text{find}\,\,x^*\in\HH\,\,\text{such that}\\
\brk3{\exi\mathsf{z}\in\mathsf{W}^{-1}\brk3{
{-}\int_{\Omega}\LW^*\brk!{x^*(\omega)}
\mu(d\omega)}}(\forallmu\omega\in\Omega)\,\,
\LW\mathsf{z}\in\BW^{-1}\brk!{x^*(\omega)}
+\gamma x^*(\omega).
\end{multline}
As in the special case discussed in \cite[Proposition~4.1]{Siim22},
which is set in the context of Example~\ref{ex:1}\ref{ex:1i+},
the inclusion \eqref{e:w7q3} can be shown to be an
exact relaxation of the inclusion problem
\begin{equation}
\text{find}\,\,\mathsf{z}\in\zer\mathsf{W}\,\,\text{such that}\,\,
(\forallmu\omega\in\Omega)\,\,
\moyo{\BW}{\gamma}(\LW\mathsf{z})=\mathsf{0}
\end{equation}
or, equivalently, of the so-called split common zero problem
\begin{equation}
\label{e:w3i0}
\text{find}\,\,\mathsf{z}\in\zer\mathsf{W}\,\,\text{such that}\,\,
(\forallmu\omega\in\Omega)\,\,
\mathsf{0}\in\BW(\LW\mathsf{z})
\end{equation}
in the sense that, if \eqref{e:w3i0} has solutions,
they are the same as those of \eqref{e:w7q3}. If we further
specialize to the case when $\mu$ is a probability measure,
$\mathsf{W}=\mathsf{0}$, and for every $\omega\in\Omega$,
$\HW=\mathsf{G}$, $\LW=\Id_{\mathsf{G}}$, and
$\BW=N_{\mathsf{C}_{\omega}}$,
where $\mathsf{C}_{\omega}$ is a nonempty closed convex subset
of $\mathsf{G}$, then \eqref{e:w3i0} collapses to the stochastic 
convex feasibility problem of \cite{Butn95}.
\end{example}

\end{document}